\documentclass[lettersize,journal]{IEEEtran}

\usepackage{amsmath,amsfonts}
\usepackage{array}
\usepackage[caption=false,font=normalsize,labelfont=sf,textfont=sf]{subfig}

\usepackage{textcomp}
\usepackage{stfloats}
\usepackage{url}
\usepackage{verbatim}
\usepackage{graphicx}
\usepackage{balance}

\usepackage{amsmath,amssymb,amsfonts}
\usepackage{algorithm}
\usepackage{algorithmicx}
\usepackage{algpseudocode}
\usepackage{graphicx}
\usepackage{textcomp}
\usepackage{url}
\usepackage{color}
\usepackage{bm}

\usepackage{lipsum}
\usepackage{epstopdf}
\usepackage{booktabs}
\usepackage{amsopn}
\DeclareMathOperator{\diag}{diag}
\DeclareMathOperator{\prox}{prox}

\usepackage{amsthm}

\usepackage{enumerate}
\usepackage{cancel}
\usepackage{mathrsfs}
\usepackage{mathdots}
\usepackage{euscript}
\usepackage{amscd}
\usepackage{placeins}

\usepackage{multirow}
\usepackage{tikz}
\usetikzlibrary{arrows}
\usepackage{verbatim}

\newtheorem{theorem}{Theorem}

\newtheorem{lemma}{Lemma}
\newtheorem{proposition}{Proposition}
\newtheorem{assumption}{Assumption}

\theoremstyle{definition}
\newtheorem{definition}{Definition}

\theoremstyle{remark}
\newtheorem{remark}{Remark}

\newcommand{\nn}{\nonumber}

\newcommand{\bmat}{\left[ \begin{matrix}}
	\newcommand{\emat}{\end{matrix} \right]}
\newcommand{\innerprod}[2]{\langle{#1},\,{#2}\rangle}

\DeclareMathOperator{\trace}{tr}

\DeclareMathOperator{\argmax}{argmax}
\DeclareMathOperator{\argmin}{argmin}

\newcommand{\Rbb}{\mathbb R}

\newcommand{\Nbb}{\mathbb N}
\newcommand{\xb}{\mathbf  x}
\newcommand{\yb}{\mathbf  y}

\newcommand{\wb}{\mathbf  w}

\newcommand{\ub}{\mathbf  u}

\newcommand{\tb}{\mathbf t}

\newcommand{\zerob}{\mathbf 0}

\newcommand{\Ab}{\mathbf A}
\newcommand{\Bb}{\mathbf B}

\newcommand{\Db}{\mathbf D}

\newcommand{\Ib}{\mathbf I}
\newcommand{\Jb}{\mathbf J}

\newcommand{\Lb}{\mathbf L}

\newcommand{\Pb}{\mathbf P}
\newcommand{\Qb}{\mathbf Q}

\newcommand{\Sb}{\mathbf S}
\newcommand{\Tb}{\mathbf T}
\newcommand{\Ub}{\mathbf U}
\newcommand{\Vb}{\mathbf V}
\newcommand{\Yb}{\mathbf Y}
\newcommand{\Zb}{\mathbf Z}
\newcommand{\Xb}{\mathbf  X}
\newcommand{\Wb}{\mathbf W}

\newcommand{\mub}{\boldsymbol{\mu}}

\newcommand{\Gammab}{\boldsymbol{\Gamma}}
\newcommand{\Sigmab}{\boldsymbol{\Sigma}}

\newcommand{\Lambdab}{\boldsymbol{\Lambda}}
\newcommand{\Thetab}{\boldsymbol{\Theta}}
\newcommand{\Phib}{\boldsymbol{\Phi}}

\DeclareMathOperator{\sign}{sign}

\newcommand{\Dscr}{\mathscr{D}}

\newcommand{\Scal}{\mathcal{S}}
\newcommand{\Lcal}{\mathcal{L}}

\newcommand{\Dcal}{\mathcal{D}}

\newcommand{\Ncal}{\mathcal{N}}

\newcommand{\SNR}{\mathrm{SNR}}
\newcommand{\KL}{\mathrm{KL}}
\newcommand{\F}{\mathrm{F}}

\newcommand{\const}{\text{const.}}

\newcommand{\lyn}[1]{{\color{black} #1}}
\newcommand{\vrm}{\mathrm{v}}

\begin{document}
\title{$\ell_0$ Factor Analysis: A P-Stationary Point Theory}
\author{Linyang Wang, Bin Zhu, and Wanquan Liu
\thanks{The authors are with the School of Intelligent Systems Engineering, Sun Yat-sen University, Gongchang Road 66, 518107 Shenzhen, China (emails: \texttt{wangly227@mail2.sysu.edu.cn, \{zhub26, liuwq63\}@mail.sysu.edu.cn}).}
\thanks{This work was supported in part by Shenzhen Science and Technology Program (Grant No.~202206193000001-20220817184157001), the Fundamental Research Funds for the Central Universities, and the ‘‘Hundred-Talent Program’’ of Sun Yat-sen University.}}

\maketitle

\begin{abstract}
Factor Analysis is a widely used modeling technique for stationary time series which achieves dimensionality reduction by revealing a hidden low-rank plus sparse structure of the covariance matrix.	
Such an idea of parsimonious modeling has also been important in the field of systems and control.
In this article, a nonconvex nonsmooth optimization problem involving the $\ell_{0}$ norm is constructed in order to achieve the low-rank and sparse additive decomposition of the sample covariance matrix.
We establish the existence of an optimal solution and characterize these solutions via the concept of proximal stationary points. 
Furthermore, an ADMM algorithm is designed to solve the $\ell_{0}$ optimization problem, and a subsequence convergence result is proved under reasonable assumptions. 
Finally, numerical experiments demonstrate the effectiveness of our method in comparison with some alternatives in the literature.
\end{abstract}


\section{Introduction}

Factor Analysis (FA) is a classic topic which originates in psychology \cite{spearman1904general,ledermann1937rank,shapiro1982rank} and is studied in econometrics \cite{watson1983alternative,picci1989parametrization,forni2001generalized,bai2002determining} with modern developments in diverse fields such as system identification, signal processing, statistics, and machine learning, cf.~\cite{erson1993identification,heij1997system,deistler2007modelling,lam2012factor,bottegal2014modeling,bertsimas2017certifiably,zorzi2017sparse,ciccone2018factor,falconi2024robust} and the references therein. A standard model for FA can be found e.g., in \cite[Sec.~7.7]{hastie2015statistical} and reads as
\begin{equation}\label{generate_model}
	\yb_i = \mub + \Gammab \ub_i + \wb_i,\quad i=1,2,\dots,N,
\end{equation}
where 
\begin{itemize}
	\item $\yb_i\in\Rbb^p$ is an observed vector, 
	\item $\mub\in\Rbb^p$ is a mean vector, 
	\item $\Gammab\in\Rbb^{p\times r}$ is a deterministic ``factor loading'' matrix with linearly independently columns, 
	\item the random vector $\ub_i\sim \Ncal(0,\Ib_r)$ stands for the hidden factors, 
	\item and $\wb_i\sim \Ncal(0, \hat \Sb)$ is the additive noise with an \emph{unknown} covariance matrix $\hat \Sb$ independent of $\ub_i$.
\end{itemize}
The model is widely used for linear dimensionality reduction due to the typical situation that $r\ll p$. Suppose that we have obtained $N$ i.i.d.~samples $\yb_i$ from the model. 
\lyn{The problem is about estimating the symmetric matrix $\hat\Lb := \Gammab\Gammab^\top\in\Rbb^{p\times p}$ of rank $r$  and the noise covariance matrix $\hat\Sb$.} From the viewpoint of system identification, this is a \emph{static} identification problem in which the output $\yb_i$ depends solely on the current input $\ub_i$.

Assume for simplicity that the mean vector $\mub=\zerob$. According to the independence assumption, the covariance matrix of $\yb_i$ admits an additive decomposition as
\begin{equation}\label{Sigma_decomp}
	\Sigmab = \Gammab\Gammab^\top+\hat\Sb = \hat\Lb + \hat\Sb.
\end{equation}
In the special case where $\hat\Sb=\sigma^2 \Ib_p$, the problem reduces to the standard \emph{Principal Component Analysis} (PCA). A slightly relaxed assumption requires $\hat\Sb$ to be diagonal, as is typically done in FA. Then \eqref{Sigma_decomp} can be interpreted as a kind of ``low-rank plus sparse'' matrix decomposition. In practice, the true covariance matrix $\Sigmab$ is very often not available and instead must be estimated from the observed samples. For example, a usual averaging scheme gives
\begin{equation}\label{sample_cov}
	\check{\Sigmab} = \frac{1}{N} \sum_{i=1}^{N} \yb_i \yb_i^\top = \hat\Lb + \hat\Sb + \Wb,
\end{equation}
where $\Wb$ is a residual matrix. The FA problem now becomes estimating the low-rank $\hat\Lb$ and the sparse $\hat\Sb$ from the noisy estimate $\check{\Sigmab}$ of the covariance matrix.

Following this type of thinking, there has been extensive research on the so-called \emph{robust PCA}, see e.g., \cite{candes2011robust,chandrasekaran2011rank,hsu2011robust,agarwal2012noisy,wen2019robust,chen2021bridging}, where convex optimization methods are employed to separate the low-rank matrix from the sparse one, and techniques from compressed sensing are used to establish strong recovery guarantees.
Applications include foreground/background extraction in video surveillance, and face recognition, etc.
However, we must point out that these works mostly utilize the $\ell_1$ norm on the component $\Sb$ in order to enforce sparsity, while the most natural choice is the \emph{$\ell_0$ norm}\footnote{Notice that the $\ell_0$ norm is \emph{not} a \emph{bona fide} norm as it obviously violates the absolute homogeneity. It is still termed ``norm'' in this paper for simplicity.} which counts the number of nonzero entries in a matrix. Moreover, they use the squared Frobenius norm to quantify the mismatch between candidate $\Lb+\Sb$ and the estimate $\check{\Sigmab}$. In this way, no special attention is paid to covariance matrices, and even rectangular matrices are allowed.

In the recent paper \cite{ciccone2018factor}, the FA problem is formulated via constrained convex optimization involving the \emph{Kullback--Leibler (KL) divergence} between two positive definite matrices.
The KL divergence here is in fact a special case of the object with the same name, and can be interpreted as a pseudodistance between two zero-mean multivariate normal distributions where the two positive definite matrices are the covariance matrices, see e.g., \cite{LP15}.
It has been extensively used in information theory, statistics \cite{cover2005elements}, and system identification, notably in spectral estimation \cite{Georgiou-L-03,PavonF-06,FPR-08,FRT-11} and graphical model learning \cite{songsiri2010topology,avventi2013arma,zorzi2015ar,zorzi2017sparse,ciccone2020learning,alpago2022scalable,you2022generalized,alpago2023data,you2024sparse}. Motivated by the work in \cite{ciccone2018factor}, in the current paper we shall remove the usual constraint for $\Sb$ to be diagonal, and instead require $\Sb$ to be sparse as measured by the $\ell_0$ norm. The resulting optimization problem is, however, nonconvex and nonsmooth due to the discrete nature of the $\ell_0$ norm. Drawing inspirations from recent developments in the $\ell_0$ sparse optimization, cf.~e.g., \cite{nikolova2013description,marjanovic2015l0,pan2015solutions,bertsimas2017certifiably,zhang2021optimality,zhou2021newton,zhou2021global,zhou2021quadratic,wang2021support,shi2023admm,bertsimas2023sparse}, we approach the problem via the theory of \emph{proximal stationary points} which can be viewed as a nonsmooth version of the first-order Karush-Kuhn-Tucker (KKT) conditions. Then the optimality conditions are integrated into an Alternating Direction Method of Multipliers (ADMM) for the numerical solution of the optimization problem, and a subsequence convergence result is provided under certain assumptions.
Simulations on synthetic and real data show that our method is able to find the number of hidden factors, i.e., the rank of $\hat \Lb$, in a robust way.

\subsection{Organization}

The remainder of our paper is organized as follows.
Using the $\ell_{0}$ norm in conjunction with the KL divergence, Section \ref{sec:prob} constructs the nonconvex optimization model for low-rank plus sparse matrix decomposition and establishes the existence of a solution. Section \ref{sec:optimal} provides the optimality theory that connects an optimal solution to a P-stationary point. Section~\ref{sec:alg} presents an ADMM algorithm to find a solution and its local convergence analysis. Section \ref{sec:numerical} conducts numerical experiments to assess the effectiveness of the proposed algorithm through its application to both synthetic and real datasets. Finally, we conclude this work in Section \ref{sec:conc}.

\subsection{Notation}
Throughout this article, we use bold uppercase letters and bold lowercase letters to represent matrices and vectors, respectively. For two square matrices $\Ab=[a_{ij}]$ and $\Bb=[b_{ij}]$ of the same size, we write the inner product as $\innerprod{\Ab}{\Bb} := \trace(\Ab\Bb^\top)$. The symbols $(\cdot)^\top$ and $\|\cdot\|_\F$ denote transpose and the Frobenius norm, respectively. The $\ell_0$ norm $\|\cdot\|_0$ counts the number of nonzero entries in a matrix or a vector. 
If $\Ab$ is positive semidefinite, we write $\Ab\succeq 0$, and $\Ab\succ 0$ means that $\Ab$ is positive definite.

\section{Problem Formulation and Existence of a Solution}\label{sec:prob}

In this paper, we mainly consider the following optimization problem:
\begin{subequations}\label{opt_formulation}
	\begin{align}
		& \underset{\Sigmab, \mathbf{L}, \mathbf{S}}{\operatorname{min}} & & \operatorname{tr} (\mathbf{L})+C\|\mathbf{S}\|_0+\mu \mathcal{D}_{\mathrm{KL}}(\mathbf{\Sigma}\| \lyn{\check{\mathbf{\Sigma}}}) \label{obj_func} \\
		& \ \ \text{s.t.} & &\mathbf{L}\succeq 0, \label{constraint_positive_L} \\
		& & & \mathbf{S} \succeq 0, \label{constraint_positive_S} \\
		& & & \mathbf{\Sigma}=\mathbf{L}+\mathbf{S}\succ 0  \label{constraint_additive},
	\end{align}
\end{subequations}
where $\Lb$, $\Sb$, $\Sigmab$ and the sample covariance matrix $\check{\Sigmab}$ are all square of size $p$,  
$$\Dcal_{\KL}(\Sigmab||\check{\Sigmab}) := \log\det(\Sigmab^{-1} \check{\Sigmab}) + \trace(\Sigmab\check{\Sigmab}^{-1}) -p$$
is the KL divergence between two positive definite matrices, and $C$ and $\mu$ are tuning parameters. The first term $\trace(L)$ in the objective function is equal to the \emph{nuclear norm} $\|L\|_\star$ which is a convex surrogate of the rank function \cite{fazel2002matrix}. 
Throughout this article, we assume that $\check{\Sigmab}\succ 0$ so that we can take its inverse in the definition of the KL divergence.
To get a more concise optimization model, we substitute the variable $\Sigmab$ in \eqref{opt_formulation} with $\Lb+\Sb$.
Then we get the following equivalent form
\begin{subequations}\label{optim_model_2}
	\begin{align}
		& \underset{\mathbf{L},\, \mathbf{S}}{\min} & & F(\mathbf{L},\mathbf{S}) := f(\mathbf{L},\mathbf{S})+ C\Vert \mathbf{S}\Vert _0 \\
		& \ \text{s.t.} & & (\Lb, \Sb)\in\Dscr, \label{feasi_constraint}
	\end{align}
\end{subequations}
where the feasible set
\begin{equation}\label{feasible_set}
	\Dscr:=\{(\Lb, \Sb) : \Lb\succeq 0,\ \Sb\succeq 0,\ \text{and}\ \Lb+\Sb\succ 0 \},
\end{equation}
and
\begin{subequations}\label{f_specific}
	\begin{align}
		f(\mathbf{L},\mathbf{S}) & := \trace(\mathbf{L}) +
		\mu \left[\mathcal{D}_{\mathrm{KL}}(\Lb+\Sb\| \check{\mathbf{\Sigma}}) \underbrace{-\log\det\check{\Sigmab}+p}_{\const} \right] \label{f_struc_01} \\
		& =\trace(\mathbf{L})+\mu\left\{ \operatorname{tr}\left[(\mathbf{L}+\mathbf{S}) \check{\mathbf{\Sigma}}^{-1}\right] - \log \operatorname{det}(\mathbf{L}+\mathbf{S})\right\}. \label{f_struc_02}
	\end{align}
\end{subequations}

\begin{remark}
	The paper \cite{ciccone2018factor} considered the following optimization problem for FA:
	\begin{equation}\label{opt_CFZ}
	\begin{aligned}
	& \underset{\Sigmab,\, \Lb,\, \Sb}{\text{min}}
	& & \trace(\Lb) \\
	& \ \ \text{s.t.}
	& & \Lb, \Sb \succeq 0 \\
	& & & \Sb \text{ is diagonal} \\
	&&& \Sigmab = \Lb + \Sb \succ 0 \\
	&&& \Dcal_{\KL}(\Sigmab||\check{\Sigmab})\leq \delta/2.
	\end{aligned}
	\end{equation}
    The idea is to utilize the KL divergence to constrain the candidate covariance matrix $\Sigmab$ in a ``neighborhood'' of $\check{\Sigmab}$ controlled by the tuning parameter $\delta$.
	The interest in this constrained version instead of the regularized version similar to \eqref{opt_formulation} lies in the fact that a probabilistic strategy can be adopted to select the parameter $\delta$ via analysis with the \emph{Gaussian Orthogonal Ensemble}, so that ad-hoc procedures for parameter selection such as the cross validation is not needed. We assume that the computational power does not pose serious limitations for us, and hence do not insist on this constrained version in the current work.
\end{remark}

	It is worth noting that the nonconvexity of the objective function $F(\mathbf{L},\mathbf{S})$ is solely attributed to the term $\|\Sb\|_0$.

More precisely, we have the next result.
	
	\begin{proposition}\label{prop_strict_convex}
		The function $f$ in \eqref{f_specific} is smooth and jointly strictly convex in $(\Lb,\Sb)$.
	\end{proposition}
\begin{proof}
	The additive structure of $f$ in \eqref{f_specific} can be simplified as $f(\Lb,\Sb) = h(\Lb) + \mu g(\Lb+\Sb)+\const$, where $h(\Lb)=\trace(\Lb)$ is linear and $g(\cdot) = \Dcal_{\KL}(\cdot\,\| \check{\Sigmab})$ is smooth and strictly convex \cite{LP15}. The claim of the proposition follows directly from the fact that the term $g(\Lb+\Sb)$ is smooth and strictly convex. The smoothness is trivial as long as $\Lb+\Sb\succ0$. The strict convexity reduces simply to definition checking.
\end{proof}

\begin{remark}
	\lyn{The problem formulation in \cite{ciccone2018factor} imposes the diagonal constraint on the sparse structure of $\Sb$ which can be understood as prior information about the noise. In the absence of such prior information, our formulation \eqref{opt_formulation}, which only requires $\Sb$ to be sparse, is expected to be more flexible and realistic.}
\end{remark}

The next result concerns the existence of a solution to \eqref{optim_model_2} whose proof needs some additional lemmas.

\begin{theorem}\label{thm_exist}
	The optimization problem \eqref{optim_model_2} admits a solution in $\Dscr$ and the set of all minimizers is bounded.	
\end{theorem}

\begin{lemma}\label{lem_lower_semiconti}
	The objective function $F(\Lb,\Sb)$ in \eqref{optim_model_2} is lower semicontinuous.
\end{lemma}
\begin{proof}
	In view of Proposition~\ref{prop_strict_convex}, we only need to show that the $\ell_{0}$ norm $\|\Sb\|_0$ is lower semicontinuous. Let $\Sb\succeq 0$ be arbitrary. By the definition of lower semicontinuity, we have to establish the fact that for any $\varepsilon>0$, there exists $\delta>0$ such that
	\begin{equation}\label{inequal_semi_conti}
	\|\Tb-\Sb\|_\F<\delta \implies \|\Tb\|_0>\|\Sb\|_0-\varepsilon.
	\end{equation}
	Indeed, when $\delta$ is small enough, all the nonzero entries of $\Sb$ must remain nonzero in $\Tb$, and we have
	\begin{equation}\label{inequal_S_zero_norm}
	\|\Tb\|_0\geq \|\Sb\|_0.
	\end{equation} 
	\emph{A fortiori}, \eqref{inequal_semi_conti} must hold and the proof is complete.	
\end{proof}

\begin{lemma}\label{lem_diverg}
	Suppose that there is a sequence $\{(\Lb^k,\Sb^k)\}_{k\geq 1}\subset \Dscr$ such that $\|\Lb^k\|_\F$ or $\|\Sb^k\|_\F$ tends to infinity as $k\to\infty$. Then the sequence of function values $\{F(\Lb^{k},\Sb^{k})\}_{k\geq 1}$ also goes unbounded.	
\end{lemma}

\begin{proof}
	Since the term $\|\Sb\|_0$ always remains bounded, we only need to show the divergence of the sequence $\{f(\Lb^k, \Sb^k)\}_{k\geq 1}$. Suppose first that $\|\Lb^k\|_\F\to\infty$ as $k\to\infty$. Then the largest eigenvalue of $\Lb^k$ must also tends to infinity due to norm equivalence. Consequently, we have $\trace (\Lb^k)\to \infty$. For other terms in \eqref{f_struc_01}, we know $\Dcal_{\KL}(\Lb^k+\Sb^k\|\check{\Sigmab})\geq 0$ (bounded from below), see e.g., \cite{LP15}, and the rest are constants. Hence $f(\Lb^k,\Sb^k)$ necessarily diverges.
	
	Next suppose that $\|\Sb^k\|_\F\to\infty$ as $k\to\infty$. Let us write $\Sigmab^k=\Lb^k+\Sb^k$ for simplicity. Obviously, in this case the largest eigenvalue of $\Sigmab^k$ must tend to infinity. According to \eqref{f_struc_02}, we only need to compare the two terms in the brace, namely $\trace(\Sigmab^k\check{\Sigmab}^{-1}) - \log\det\Sigmab^k$. We shall use the spectral decomposition $\Sigmab^k=\Ub^k\Lambdab^k(\Ub^k)^\top$ where $\Ub^k$ is orthogonal and $\Lambdab^k=\diag\{\lambda_{k,1},\dots, \lambda_{k,p}\}$ is the diagonal matrix of eigenvalues.
	Likewise, we have $\check{\Sigmab}=\hat\Ub\hat\Lambdab\hat\Ub^\top$. Then it is not difficult to derive that
    \begin{subequations}\label{trace_minus_logdet}
	\begin{align}
		\trace(\Sigmab^k\check{\Sigmab}^{-1}) - \log\det\Sigmab^k & = \trace(\Xb\Lambdab^k\Xb^\top) -\log\det\Sigmab^k \nn \\
	 & = \sum_{i = 1}^p (c_i\lambda_{k,i}-\log\lambda_{k,i}), \label{sum_compare}
	\end{align}
	\end{subequations}
	where $\Xb=\hat\Lambdab^{-1/2}\hat{\Ub}^\top\Ub^k$, and $c_i$ is the norm squared of the $i$-th column of $\Xb$. Clearly, each $c_i$ is bounded away from zero because of the structure of $\Xb$. Consider the function $cx-\log x$ defined for $x>0$ with a parameter $c>0$. It is strictly convex on the positive semiaxis and has a minimum value of $1+\log c$ at $x=1/c$. Consequently, each summand in \eqref{sum_compare} is bounded from below and $c_1\lambda_{k,1}-\log\lambda_{k,1}\to\infty$ because $\lambda_{k,1}\to \infty$ as argued above.
	Therefore, $f(\Lb^k,\Sb^k)$ must also diverge in this case.
\end{proof}

\begin{lemma}\label{lem_boundary}
	Suppose that we have a convergent sequence $\{(\Lb^k, \Sb^k)\}_{k\geq 1}\subset \Dscr$ which tends to some $(\tilde\Lb, \tilde\Sb)$ as $k\to \infty$ such that $\tilde{\Lb}+\tilde{\Sb}\succeq 0$ but is singular. Then the sequence of function values $\{F(\Lb^{k},\Sb^{k})\}_{k\geq 1}$ tends to infinity.	
\end{lemma}

\begin{proof}
	It is clear that formula \eqref{trace_minus_logdet} is still valid for $\Sigmab^k=\Lb^k+\Sb^k$. Now due to convergence, each eigenvalue $\lambda_{k,i}$ must remain bounded. It is easy to observe that $c_i\lambda_{k,i}-\log\lambda_{k,i}\to\infty$ when $\lambda_{k,i}\to 0$. This implies that when $\tilde{\Sigmab}=\tilde{\Lb}+\tilde{\Sb}$ is singular, the sum in $\eqref{sum_compare}$ also diverges following the argument in the proof of Lemma~\ref{lem_diverg}.
	As a result, $f(\Lb^k, \Sb^k)$, and hence $F(\Lb^{k},\Sb^{k})$, must go to infinity as $k\to \infty$.
\end{proof}

We are now ready to prove Theorem~\ref{thm_exist}.

\begin{proof}[Proof of Theorem~\ref{thm_exist}]
	
	Take an arbitrary feasible point $(\Lb^0, \Sb^0)$ of \eqref{optim_model_2}, and let $\beta:=F(\Lb^0, \Sb^0)$. Consider the sublevel set
	\begin{equation}\label{sublevel_set}
	F^{-1}(-\infty, \beta] := \{(\Lb, \Sb)\in\Dscr : F(\Lb, \Sb)\leq \beta\}
	\end{equation}
	which is obviously nonempty because it contains $(\Lb^0, \Sb^0)$. Suppose that we have a convergent sequence $\{(\Lb^k, \Sb^k)\}_{k\geq 1}\subset F^{-1}(-\infty, \beta]$ which tends to some $(\tilde\Lb, \tilde\Sb)$. We first exclude the possibility that $\tilde{\Lb}+\tilde{\Sb}\succeq 0$ but is \emph{singular}. Indeed, if that were true, then by Lemma~\ref{lem_boundary}, $F(\Lb^k, \Sb^k)$ would eventually surpass $\beta$ which means that $(\Lb^k, \Sb^k)$ would step out of the sublevel set, a contradiction. Therefore, the limit point must be such that $\tilde{\Lb}+\tilde{\Sb}\succ 0$ and still feasible, that is, $(\tilde{\Lb}, \tilde{\Sb})\in\Dscr$. Furthermore, Lemma~\ref{lem_lower_semiconti} implies that the sublevel set is \emph{closed}. 
	
	In addition, one can argue that the sublevel set is also \emph{bounded}. To see this, suppose the contrary, i.e., the sublevel set were unbounded. Then according to Lemma~\ref{lem_diverg}, there would exist a sequence $\{(\Lb^k, \Sb^k)\}_{k\geq 1}\subset F^{-1}(-\infty, \beta]$ such that the function value $F(\Lb^k, \Sb^k)$ diverges, leading to a contradiction similar to the previous one. 
	
	Due to finite dimensionality, the sublevel set \eqref{sublevel_set} is seen to be \emph{compact}. Finally, appealing to the extreme value theorem of Weierstrass, we conclude that a minimizer of $F(\Lb, \Sb)$ exists in $F^{-1}(-\infty, \beta]$, and the boundedness claim comes from the compactness of the sublevel set.
\end{proof}

\section{Optimality Theory}\label{sec:optimal}

In this section, we shall characterize the optimal solutions to \eqref{optim_model_2}.
To begin with, we shall review some well known results concerning the $\ell_0$ proximal operator, see e.g., \cite{blumensath2008iterative}, \cite{blumensath2009iterative} where it is called \emph{hard thresholding operator}.

\subsection{$\ell_0$ proximal operator}

We first discuss the scalar case. Given a parameter $\gamma>0$, the proximal operator of $C|\cdot|_0$ is defined as
\begin{equation}\label{prox_scalar_def}
\prox_{\gamma C |\cdot|_0} (s) := \underset{v\in\Rbb}{\argmin}\ C|v|_0 + \frac{1}{2\gamma} (v-s)^2.
\end{equation}
The proof of the next lemma is a simple exercise and is therefore omitted.
\begin{proposition}
	The solution to the proximal operator in \eqref{prox_scalar_def} is expressed as
	\begin{equation}\label{prox_scalar_sol}
	\prox_{\gamma C |\cdot|_0} (s) = \begin{cases}
	0, & \text{if}\ |s| < \sqrt{2\gamma C} \\
	0\ \text{or}\ s, & \text{if}\ |s| = \sqrt{2\gamma C} \\
	s, & \text{if}\ |s| > \sqrt{2\gamma C}.
	\end{cases}
	\end{equation}
\end{proposition}

In order to avoid ambiguity, we choose to absorb the middle case of \eqref{prox_scalar_sol} into the first case. In terms of the formula, we use 
\begin{equation}\label{prox_scalar_sol1}
\prox_{\gamma C |\cdot|_0} (s) = \begin{cases}
0, & \text{if}\ |s| \leq \sqrt{2\gamma C} \\
s, & \text{if}\ |s| > \sqrt{2\gamma C}.
\end{cases}
\end{equation}
in the remaining part of this paper. The graph of such a function is drawn in Fig.~\ref{fig:prox_operat}.

\begin{figure}
	\centering
	\begin{tikzpicture}
	\draw[line width=0.5pt][->](-2,0)--(2,0)node[left,below,font=\tiny]{$s$};
	\draw[line width=0.5pt][->](0,-2)--(0,2);
	\node[below,font=\tiny] at (0.09,0.1){0};
	\draw[color=red,thick,smooth][-](-0.8,0)--(0.8,0);
	\draw[color=red,fill=red,smooth](0.8,0)circle(0.04);
	\draw[color=red,fill=red,smooth](-0.8,0)circle(0.04);
	\draw[color=red,fill=white,smooth](0.8,0.8)circle(0.04);
	\draw[color=red,fill=white,smooth](-0.8,-0.8)circle(0.04);
	\node[left,below,font=\tiny]at(0.8,0){$_{\sqrt{2\gamma C}}$};
	\draw[color=red][dashed] (0.8,0)--(0.8,0.8); 
	\draw[color=red,thick,smooth][-](0.8,0.8)--(1.8,1.8);
	\node[left,above,font=\tiny]at(-0.8,0){$_{-\sqrt{2\gamma C}}$};
	\draw[color=red][dashed] (-0.8,0)--(-0.8,-0.8); 
	\draw[color=red,thick,smooth][-](-0.8,-0.8)--(-1.8,-1.8);
	\node[above,font=\tiny] at(1.2,1.7) {$\quad\mathrm{prox}_{\gamma C|\cdot|_0}(s)$};
	\end{tikzpicture}
	\caption{The $\ell_{0}$ proximal operator on the real line.}
	\label{fig:prox_operat}
\end{figure}

The matricial version of \eqref{prox_scalar_def} is then not difficult to derive. Indeed, for $\Sb\in\Rbb^{m\times n}$ we have
\begin{equation}\label{prox_matrix}
\begin{aligned}
\prox_{\gamma C \|\cdot\|_0} (\Sb) & := \underset{\Vb\in\Rbb^{m\times n}}{\argmin}\ C\|\Vb\|_0 + \frac{1}{2\gamma} \|\Vb-\Sb\|_\F^2 \\
 & = \underset{\Vb\in\Rbb^{m\times n}}{\argmin}\ \sum_{i,j} \left[ C|v_{ij}|_0 + \frac{1}{2\gamma} (v_{ij}-s_{ij})^2 \right].
\end{aligned}
\end{equation}
Since the objective function above can be decoupled for each element $s_{ij}$, the solution to the matricial $\ell_0$ proximal operator can be written as
\begin{equation}\label{prox_elementsise}
\left[ \prox_{\gamma C \|\cdot\|_0} (\Sb) \right]_{ij} = \prox_{\gamma C |\cdot|_0} (s_{ij}).
\end{equation}
In plain words, the outcome of the proximal operator is obtained by elementwise applications of the scalar solution \eqref{prox_scalar_sol1}.

\subsection{General optimization model}

In this subsection and the next one, we work on the abstract optimization model \eqref{optim_model_2} and temporarily forget about the specific functional form of $f$ in \eqref{f_specific}. More precisely, we make the following assumptions.

\begin{assumption}\label{assump_convex}
	The function $f$ in \eqref{optim_model_2} is jointly strictly convex in $(\Lb,\Sb)$.
\end{assumption}

\begin{assumption}\label{assump_grad_Lipschitz}
	The gradient of $f$ is Lipschitz continuous with a Lipschitz constant $K>0$ on the sublevel set \eqref{sublevel_set}, that is, for any $(\Lb^1, \Sb^1)$ and $(\Lb^2, \Sb^2)$ in $F^{-1}(-\infty, \beta]$, it holds that
	\begin{equation}
	\left\| \bmat \nabla_\Lb f(\Lb^1, \Sb^1) \\ \nabla_\Sb f(\Lb^1, \Sb^1) \emat - \bmat \nabla_\Lb f(\Lb^2, \Sb^2) \\ \nabla_\Sb f(\Lb^2, \Sb^2) \emat \right\|_\F 
	\leq K \left\|\bmat \Lb^1\\ \Sb^1\emat - \bmat\Lb^2\\ \Sb^2\emat\right\|_\F
	\end{equation}
	where the gradient operators $\nabla_\Lb f(\Lb, \Sb)$ and $\nabla_\Sb f(\Lb, \Sb)$ are identified with suitable symmetric matrices. Here the superscripts are simply labels for the variables instead of usual exponents.
\end{assumption}

The next result is straightforward, and it also shows that the specific problem \eqref{optim_model_2} with $f$ in \eqref{f_specific} is just a special case for the optimality theory provided in this section.

\begin{proposition}\label{assumption_satisifies}
	The specific function $f$ in \eqref{f_specific} satisfies both Assumptions~\ref{assump_convex} and \ref{assump_grad_Lipschitz}.
\end{proposition}

\begin{proof}
	The strict convexity of $f$ has been shown in Proposition~\ref{prop_strict_convex}.	
    To show that $f$ in \eqref{f_specific} has a Lipschitz continuous gradient, we use a previous result claiming that $f$ has a compact sublevel set, see the proof of Theorem~\ref{thm_exist}. More precisely, it is well known that a smooth function ($f$ here in view of Proposition~\ref{prop_strict_convex} again) is Lipschitz continuous on any compact convex set, which yields the assertion.
\end{proof}

\subsection{P-stationarity conditions}

The Lagrangian of the optimization problem  \eqref{optim_model_2}  can be expressed as 
\begin{equation}
	\Lcal(\Lb,\Sb;\Lambdab,\Thetab) = \ f(\Lb,\Sb) +C\|\Sb\|_0 - \innerprod{\Lambdab}{\Lb} - \innerprod{\Thetab}{\Sb}
\end{equation}
where $\Lambdab \succeq 0$ and $\Thetab \succeq 0$  are the Lagrange multipliers for the constraints $\Lb\succeq 0$ and $\Sb\succeq 0$. For the moment, we ignore the strict inequality constraint $\Lb+\Sb\succ 0$. Now we introduce a generalized definition of a KKT point in nonlinear programming.

\begin{definition}[P-stationary point of \eqref{optim_model_2}]\label{def_P-stat_pt_v2}
	The pair $(\Lb^*, \Sb^*)$ is called a proximal stationary (abbreviated as P-stationary) point of \eqref{optim_model_2} if there exist $p$ by $p$ symmetric matrices $\Lambdab^*$, $\Thetab^*$ and a real number $\gamma>0$ such that the following set of conditions hold:
	\begin{subequations}\label{P-stationary-cond_v2}
	\begin{align}
		\Lb^* \succeq 0,\ \Sb^* & \succeq 0 \label{P-stat-primal-1_v2} \\
		\Lb^*+\Sb^* & \succ 0 \label{P-stat-primal-2_v2} \\
		\Lambdab^* \succeq 0,\ \Thetab^* & \succeq 0 \label{P-stat-dual_v2} \\
		\trace(\Lambdab^* \Lb^*) = 0,\ \trace(\Thetab^* \Sb^*) & = 0 \label{P-stat-complem_slack_v2} \\
		\nabla_\Lb f(\Lb^*,\Sb^*) -\Lambdab^* & = O \label{P-stat-Lagrangian-grad-1_v2} \\
		\prox_{\gamma C \|\cdot\|_0}\left[\Sb^*-\gamma (\nabla_\Sb f(\Lb^*, \Sb^*)-\Thetab^*)\right] & = \Sb^*. \label{P-stat-Lagrangian-prox_v2}
	\end{align}
\end{subequations}
\end{definition}
Formulas in \eqref{P-stationary-cond_v2} describe the usual first-order KKT conditions \cite{boyd2004convex} except for \eqref{P-stat-Lagrangian-prox_v2}. 
 Indeed \eqref{P-stat-primal-1_v2} and \eqref{P-stat-primal-2_v2} are the primal constraints, \eqref{P-stat-dual_v2} is the dual constraint, \eqref{P-stat-complem_slack_v2} is the complementary slackness condition, and \eqref{P-stat-Lagrangian-grad-1_v2} and \eqref{P-stat-Lagrangian-prox_v2} are the stationarity condition of the Lagrangian with respect to the primal variables. The unique distinction is that the nondifferentiable term $\|\Sb\|_0$ results in the proximal operator \eqref{prox_matrix}. 

\begin{theorem}\label{thm_optimality}
	Suppose that Assumptions \ref{assump_convex} and \ref{assump_grad_Lipschitz} hold. Then we have the following claims.
	\begin{enumerate}
		\lyn{
		\item If $(\Lb^*, \Sb^*)$ is a global minimizer of \eqref{optim_model_2} and the sparse component $\Sb^*$ is positive definite, then $(\Lb^*, \Sb^*)$ is a P-stationary point where the parameter $\gamma$ can take an arbitrary value in the interval $(0, {1}/{(2K)})$.
		\item  
		If $(\Lb^*, \Sb^*)\in\Dscr$ is a P-stationary point, then $(\Lb^*, \Sb^*)$ is a local minimizer of \eqref{optim_model_2}.
	}  
	\end{enumerate}
	
\end{theorem}

\begin{proof}
	\lyn{To show Point 1), let us suppose that $(\Lb^*, \Sb^*)$ is a global minimizer of \eqref{optim_model_2}, and then we need to show that the P-stationary point conditions \eqref{P-stationary-cond_v2} hold for $(\Lb^*, \Sb^*)$. }
 
\lyn{The conditions \eqref{P-stat-primal-1_v2} and \eqref{P-stat-primal-2_v2} are just feasibility. 
Then, we first fix $\Sb^*\succ0$, and consider the optimality with respect to $\Lb$.
Indeed, the term $C\Vert \mathbf{S}\Vert _0$ in \eqref{optim_model_2} now becomes a constant $C\Vert \mathbf{S}^*\Vert _0$, which implies that} 
	\begin{equation}\label{subpro_l}
		\Lb^* = \underset{\Lb\succeq 0}{\argmin}\ f(\Lb, \Sb^*) 
	\end{equation}
	where the right hand side is a smooth convex optimization problem. Hence the usual KKT conditions read \lyn{as $\Lambdab^*\succeq0$} in \eqref{P-stat-dual_v2}, $\trace(\Lambdab^* \Lb^*) = 0$ in \eqref{P-stat-complem_slack_v2},  and \eqref{P-stat-Lagrangian-grad-1_v2},
	where $\Lambdab^*$ is the Lagrange multiplier \lyn{associated to the constraint $\Lb\succeq 0$}. 
	
	Next, we fix $\Lb^*$ and derive the remaining conditions, namely  $\Thetab^*\succeq0$ in \eqref{P-stat-dual_v2}, $\trace(\Thetab^* \Sb^*) = 0$ in \eqref{P-stat-complem_slack_v2},  and \eqref{P-stat-Lagrangian-prox_v2} \lyn{by considering the optimality with respect to $\Sb$. Similar to the characterization \eqref{subpro_l},} we have
	\begin{equation}\label{subpro_s}
		\Sb^* = \underset{\substack{\Sb\succeq 0,\, \Lb^*+\Sb\succ 0}}{\argmin}\ f(\Lb^*, \Sb) + C\|\Sb\|_0.
	\end{equation}
 The usual KKT analysis brings about the conditions $\Thetab^*\succeq0$ and $\trace(\Thetab^* \Sb^*) = 0$. In fact, we must have $\Thetab^* =O$ due to the additional assumption $\Sb^*\succ 0$ and $\trace(\Thetab^* \Sb^*)=0$ holds trivially.
 Hence \eqref{P-stat-Lagrangian-prox_v2} reduces to 
 \begin{equation}\label{reduced_Lagrangian_prox}
      \prox_{\gamma C \|\cdot\|_0}\left[\Sb^*-\gamma \nabla_\Sb f(\Lb^*, \Sb^*)\right] = \Sb^*
 \end{equation}
\lyn{which remains to be established in the following}.

By global optimality \lyn{of $(\Lb^*, \Sb^*)$}, we have
	\begin{equation}\label{global_pt_i}
		f(\Lb^*, \Sb^*) + C\|\Sb^*\|_0 \leq f(\Lb^*, \Zb) + C\|\Zb\|_0,
	\end{equation}
 where $\Zb$ is an arbitrary  positive semidefinite matrix that satisfies $\Lb^*+\Zb \succ0$.
 \lyn{Let $\mathbf{\Xi}^* := \nabla_\Sb f(\Lb^*, \Sb^*)$. Then, we obtain the inequalities}
 		\begin{equation}\label{global_pt_ii}
 			\begin{aligned}
 				f(\Lb^*, \Zb) & = f(\Lb^*, \Sb^*) + \innerprod{\nabla_\Sb f(\Lb^*, \Bar{\Sb})}{\Zb-\Sb^*} \\
 				& = f(\Lb^*, \Sb^*) + \innerprod{\lyn{\mathbf{\Xi}^*}}{\Zb-\Sb^*} \\
 				& \quad +\innerprod{\nabla_\Sb f(\Lb^*, \bar{\Sb})-\lyn{\mathbf{\Xi}^*}}{\Zb-\Sb^*} \\
 				& \leq f(\Lb^*, \Sb^*) + \innerprod{\lyn{\mathbf{\Xi}^*}}{\Zb-\Sb^*} + K\|\Zb - \Sb^*\|_\F^2,
 			\end{aligned}
 		\end{equation}
 \lyn{
 where the first equality is due to the mean value theorem with $\Bar{\Sb} = t\Zb + (1-t)\Sb^*$ for some $t\in(0, 1)$, and the last inequality follows from Cauchy--Schwarz and Assumption~\ref{assump_grad_Lipschitz}.
 }
		Now, let us take a particular $\Zb = \prox_{\gamma C \|\cdot\|_0} (\Sb^*-\gamma\lyn{\mathbf{\Xi}^*})$ which is the left side of \eqref{reduced_Lagrangian_prox} with $0<\gamma<{1}/{ (2K)}$. According to the definition of the proximal operator \eqref{prox_matrix}, we have
		\begin{equation}\label{global_pt_iii}
				C\|\Zb\|_0 + \frac{1}{2\gamma} \|\Zb- (\Sb^*-\gamma\lyn{\mathbf{\Xi}^*}) \|_\F^2 \leq C\|\Sb^*\|_0 + \frac{\gamma}{2} \|\lyn{\mathbf{\Xi}^*}\|_\F^2.
			\end{equation}
		Combining the previous inequalities, we obtain the following:
		\begin{subequations}
				\begin{align}
						0 & \leq  C\|\Zb\|_0+f(\Lb^*, \Zb)- C\|\Sb^*\|_0 -f(\Lb^*, \Sb^*) \label{inequal_01} \\ 
						& \leq  C\|\Zb\|_0 -C\|\Sb^*\|_0 + \innerprod{\lyn{\mathbf{\Xi}^*}}{\Zb-\Sb^*} + K\|\Zb - \Sb^*\|_\F^2 \label{inequal_02} \\
						& = C\|\Zb\|_0 -C\|\Sb^*\|_0 + \frac{1}{2\gamma} \|\Zb-(\Sb^*-\gamma\lyn{\mathbf{\Xi}^*})\|_\F^2 \nn \\ 
						& \quad -\frac{\gamma}{2}\|\lyn{\mathbf{\Xi}^*}\|_\F^2 +(K-\frac{1}{2\gamma})\|\Zb-\Sb^*\|_\F^2  \nn \\
						& \leq (K-\frac{1}{2\gamma})\|\Zb-\Sb^*\|_\F^2 \leq 0, \label{inequal_03}
					\end{align}
			\end{subequations}
		where \eqref{inequal_01} is the same as \eqref{global_pt_i}, \eqref{inequal_02} is implied by \eqref{global_pt_ii}, \eqref{inequal_03} is a consequence of \eqref{global_pt_iii}, and the last inequality comes from the fact $K-1/(2\gamma)<0$ due to our choice of $\gamma$. Therefore, we have $\|\Zb-\Sb^*\|_\F=0$ and thus $\Zb=\Sb^*$ which is precisely \eqref{reduced_Lagrangian_prox}, a simplified version of \eqref{P-stat-Lagrangian-prox_v2}.
  
	Now let us treat Point 2). \lyn{Suppose that $(\Lb^*, \Sb^*)\in\Dscr$ is a P-stationary point in the sense of Definition~\ref{def_P-stat_pt_v2}, where $\Dscr$ is the feasible set \eqref{feasible_set}.}
	For simplicity, we write $\Phib^*=(\Lb^*,\Sb^*)$ and still $\lyn{\mathbf{\Xi}^*} = \nabla_\Sb f(\Lb^*, \Sb^*)$.
	The aim is to show that \emph{there exists a sufficiently small neighborhood
	\begin{equation}\label{Phi_star_neighborhd}
		U(\Phib^*, \delta) := \{\Phib=(\Lb, \Sb) : \|\Phib-\Phib^*\|_\F<\delta\}
	\end{equation} 
	with $\delta>0$ such that $\Phib^*$ is locally optimal in $U(\Phib^*,\delta)\cap\Dscr$}. 
	That is to say, for all $(\Lb,\Sb)\in U(\Phib^*,\delta)\cap\Dscr$, it holds that
	\begin{equation}\label{local_optim}
		f(\Lb^*, \Sb^*) + C\|\Sb^*\|_0 \leq f(\mathbf{L},\mathbf{S})+ C\Vert \mathbf{S}\Vert _0.
	\end{equation}
 
	First, we know that there exists $\delta_1>0$ such that 
 \begin{equation}\label{inequal_S_zero_norm_simil}
     \|\Sb\|_0 \geq \|\Sb^*\|_0
 \end{equation}
 holds for all $(\Lb,\Sb)\in U(\Phib^*, \delta_1)$, and the reason has been explained in the proof of Lemma~\ref{lem_lower_semiconti}.
	Then by the continuity of $f(\Lb, \Sb)$, there exists $\delta_2>0$ such that $(\Lb, \Sb)\in U(\Phib^*, \delta_2)$ implies 
	\begin{equation}\label{inequal_continu}
		|f(\Lb,\Sb)-f(\Lb^*,\Sb^*)|<C \implies f(\Lb^*,\Sb^*)-C<f(\Lb,\Sb).
	\end{equation}
	We take $\delta=\min\{\delta_1, \delta_2\}$ in \eqref{Phi_star_neighborhd}, \lyn{and proceed to prove the local optimality of $\Phib^*$, i.e., to show the inequality \eqref{local_optim}.
	
	 Let us introduce some index sets}. Let $\Nbb_p:=\{1,2,\dots,p\}$ and $\Nbb_p^2=\Nbb_p\times\Nbb_p$ be the Cartesian product providing indices for the entries of $p\times p$ matrices. Define $\Gamma_* := \{(i, j) \in\Nbb_p^2 : s^*_{ij}=0\}$ and $\overline{\Gamma}_*:=\Nbb_p^2\backslash\Gamma_*$ the complement. According to the evaluation formula \eqref{prox_scalar_sol1} for the proximal operator, we have the following observations.
	\begin{itemize}
		\item For $(i,j)\in\Gamma_*$, the elementwise version of \eqref{P-stat-Lagrangian-prox_v2} reads as
  $\prox_{\gamma C |\cdot|_0}[s^*_{ij}-\gamma (\lyn{\xi_{ij}^*}-\theta^*_{ij})] = s^*_{ij}=0$. This implies that 
		\begin{equation}
			|\gamma (\lyn{\xi_{ij}^*}-\theta^*_{ij})|\leq \sqrt{2\gamma C} \iff |\lyn{\xi_{ij}^*}-\theta^*_{ij}|\leq \sqrt{2C/\gamma}.
		\end{equation}
		\item For $(i,j)\in \overline{\Gamma}_*$ on the contrary, we have $s_{ij}^*\neq 0$ and $|s^*_{ij}-\gamma (\lyn{\xi_{ij}^*}-\theta^*_{ij})|> \sqrt{2\gamma C}$, which in view of \eqref{prox_scalar_sol1}, leads to
		\begin{equation}\label{property_ind_set_comple}
			s^*_{ij}-\gamma (\lyn{\xi_{ij}^*}-\theta^*_{ij}) = s_{ij}^* \iff\lyn{\xi_{ij}^*} = \theta^*_{ij}.
		\end{equation}
	\end{itemize}
	
	Next we work on two cases corresponding to a partition of the feasible set $\Dscr$. Define
	\begin{equation}
		\Dscr_1 := \Dscr\cap\{\Phib=(\Lb,\Sb) : s_{ij}=0\ \forall (i,j)\in\Gamma_*\}
	\end{equation}
	where the latter set in the intersection is a linear subspace.
	
	\emph{Case 1}: $\Phib\in \Dscr_1\cap U(\Phib^*,\delta)$. We can employ Assumption~\ref{assump_convex} and write down the following chain of inequalities:
	\begin{subequations}
		\begin{align}
			& f(\Lb,\Sb) -f(\Lb^*,\Sb^*) \nn \\ 
			\geq\, &  \innerprod{\nabla_\Lb f(\Lb^*,\Sb^*)}{\Lb-\Lb^*} + \innerprod{\nabla_\Sb f(\Lb^*,\Sb^*)}{\Sb-\Sb^*} \label{local_inequal_01} \\
			= \, & \innerprod{\Lambdab^*}{\Lb-\Lb^*} + \innerprod{\lyn{\mathbf{\Xi}^*}}{\Sb-\Sb^*} \label{local_inequal_02} \\
			= \, & \innerprod{\Lambdab^*}{\Lb} + \innerprod{\lyn{\mathbf{\Xi}^*}}{\Sb-\Sb^*}  \label{local_inequal_03} \\
			\geq\, & \innerprod{\lyn{\mathbf{\Xi}^*}}{\Sb-\Sb^*} - \innerprod{\Thetab^*}{\Sb} = \innerprod{\lyn{\mathbf{\Xi}^*}- \Thetab^*}{\Sb-\Sb^*} \label{local_inequal_04} \\
			= \, & \sum_{(i,j)\in \Gamma_*} (\lyn{\xi_{ij}^*} - \theta_{ij}^*)(s_{ij}-s^*_{ij})  \label{local_inequal_05} \\
			= \, &\sum_{(i,j)\in\Gamma_*} (\lyn{\xi_{ij}^*}-\theta_{ij}^*) s_{ij} =0, \label{local_inequal_06}
		\end{align}
	\end{subequations} 
	where,
	\begin{itemize}
		\item \eqref{local_inequal_01} is the first-order characterization of convexity,
		\item \eqref{local_inequal_02} results from substitutions using \eqref{P-stat-Lagrangian-grad-1_v2} and the definition of \lyn{$\mathbf{\Xi}^*$},
		\item \eqref{local_inequal_03} and \eqref{local_inequal_04} hold in view of \eqref{P-stat-complem_slack_v2}, $\innerprod{\Lambdab^*}{\Lb}\geq 0$ and $\innerprod{\Thetab^*}{\Sb}\geq 0$ since all matrices are positive semidefinite,
		\item \eqref{local_inequal_05}and \eqref{local_inequal_06} is the consequence of \eqref{property_ind_set_comple} and the definitions of $\Gamma_*$ and $\Dscr_1$.
	\end{itemize}
	In short, the reasoning above gives $f(\Lb,\Sb)\geq f(\Lb^*,\Sb^*)$ which plus \eqref{inequal_S_zero_norm_simil} yields the desired inequality \eqref{local_optim}.
	
	\emph{Case 2}: $\Phib\in (\Dscr\backslash\Dscr_1)\cap U(\Phib^*,\delta)$. In this case, there exists some $(i,j)\in\Gamma_*$ such that $s_{ij}\neq 0$ while $s_{ij}^*=0$. It follows that $\|\Sb\|_0\geq \|\Sb^*\|_0 +1$. Combining this point with \eqref{inequal_continu} leads to
	\begin{equation}
		\begin{aligned}
			f(\Lb^*,\Sb^*)+C\|\Sb^*\|_0 & \leq f(\Lb^*,\Sb^*)+C\|\Sb\|_0-C \\
			& < f(\Lb,\Sb)+C\|\Sb\|_0.
		\end{aligned}
	\end{equation}
	
	Therefore, in both cases \eqref{local_optim} holds and this completes the proof of local optimality of a P-stationary point.
\end{proof}

Theorem~\ref{thm_optimality} tells us that in order to obtain a locally optimal solution to \eqref{optim_model_2}, it suffices to find a P-stationary point, and this is the aim of the next section.

\section{Numerical Algorithm}\label{sec:alg}
 
 In this section, we propose an ADMM algorithm to solve the optimization problem \eqref{optim_model_2} and perform convergence analysis of this algorithm.

\subsection{ADMM for \eqref{optim_model_2}}\label{subsec:admm}

The presence of the log-determinant term in \eqref{optim_model_2} implicitly enforces that $\Lb+\Sb\succ 0$. To simplify the algorithm design, we temporarily disregard this constraint. In addition, we introduce two auxiliary variables $\Ub$ and $\Vb$ with equality constraints $\Lb =\Ub$ and $\Sb=\Vb$, and the indicator function to handle the remaining inequality constraints, see \cite[Sec.~5]{boyd2011distributed}. The problem \eqref{optim_model_2} can be reformulated as
\begin{equation}\label{optim_model_admm}
\begin{aligned}
& \underset{\mathbf{L},\, \mathbf{S},\, \Ub, \, \Vb}{\min}\ & &
f(\mathbf{L},\mathbf{S})+ C\Vert \mathbf{S}\Vert _0 + g(\Ub) + g(\Vb) \\
& \quad \text{ s.t. } & & \Lb=\Ub \\
& & & \Sb=\Vb, \\
\end{aligned}
\end{equation}
where the indicator function
\begin{equation}
g(\Ub) = 
\begin{cases}
0 & \text{if}\ \Ub\succeq 0 \\
+\infty & \text{otherwise}
\end{cases}.
\end{equation}
Let \lyn{$\Lambdab$} and \lyn{$\Thetab$} denote the Lagrange multipliers associated with the equality constraints in \eqref{optim_model_admm}. Then we construct the augmented Lagrangian \cite{boyd2011distributed} 
\begin{equation}\label{augment_function}
\begin{aligned}
\Lcal_\rho(\Lb,\Sb,\Ub,\Vb;\Lambdab,\Thetab) = &\ f(\Lb,\Sb) +C\|\Sb\|_0 +g(\Ub) +g(\Vb) \\
 & -\innerprod{\Lambdab}{\Lb-\Ub} + \frac{\rho}{2} \|\Lb-\Ub\|_\F^2 \\
 & -\innerprod{\Thetab}{\Sb-\Vb} + \frac{\rho}{2} \|\Sb-\Vb\|_\F^2,
\end{aligned}
\end{equation}
where $\rho>0$ is a penalty parameter. Given the current iterate $(\Lb^k,\Sb^k,\Ub^k,\Vb^k;\Lambdab^k,\Thetab^k)$, the ADMM updates each variable via a Gauss--Seidel scheme: 
\begin{subequations}\label{update_step}
\begin{align}
\Lb^{k+1} & = \underset{\Lb}{\argmin}\ \Lcal_\rho(\Lb,\Sb^k,\Ub^k,\Vb^k;\Lambdab^k,\Thetab^k) \label{subprob_L} \\
\Sb^{k+1} & = \underset{\Sb}{\argmin}\ \Lcal_\rho(\Lb^{k+1},\Sb,\Ub^k,\Vb^k;\Lambdab^k,\Thetab^k) \label{subprob_S} \\
\Ub^{k+1} & = \underset{\Ub\succeq 0}{\argmin}\ \Lcal_\rho(\Lb^{k+1},\Sb^{k+1},\Ub,\Vb^k;\Lambdab^k,\Thetab^k) \label{subprob_U} \\
\Vb^{k+1} & = \underset{\Vb\succeq0}{\argmin}\ \Lcal_\rho(\Lb^{k+1},\Sb^{k+1},\Ub^{k+1},\Vb;\Lambdab^k,\Thetab^k) \label{subprob_V} \\
\Lambdab^{k+1} & = \Lambdab^k - \rho(\Lb^{k+1}-\Ub^{k+1}) \label{dual_update_Lambda} \\
\Thetab^{k+1} & = \Thetab^k - \rho(\Sb^{k+1}-\Vb^{k+1}). \label{dual_update_Theta}
\end{align}
\end{subequations}

\begin{itemize}
	\item The $\Lb$-subproblem \eqref{subprob_L} can be simplified as
	\begin{equation}\label{subprob_L_equivalent}
		\begin{aligned}
			\Lb^{k+1} &= \underset{\Lb}{\argmin}\ f(\Lb, 	\Sb^k) -\innerprod{\Lambdab^k}{\Lb} + \frac{\rho}{2} \|\Lb-\Ub^k\|_\F^2 \\
		&= \underset{\Lb} {\argmin}  -\log\det(\Lb+\Sb^k) +\frac{\rho}{2\mu }\|\Lb-\Ub^k\|_\F^2\\
			&~~~~~~~~~~~~~ +\innerprod{\Lb}{\tfrac{1}{\mu}\Ib-\tfrac{1}{\mu}\Lambdab^{k}+\check{\Sigmab}^{-1}}.
		\end{aligned}
	\end{equation}
	By \cite[Lemma 1]{de2021graphical}, the global minimizer of \eqref{subprob_L_equivalent} has a closed-form expression
	\begin{equation}\label{solution_L}
		\Lb^{k+1}=  \frac{\mu}{2 \rho}   \Xb^k ( \sqrt{ (\Db^k)^{2}+ \frac{4\rho}{\mu}\Ib }- \Db^k )({\Xb^k})^\top -\Sb^{k},
	\end{equation}
	where $\Xb^k \Db^k({\Xb^k})^\top$ is
	 the spectral decomposition of the symmetric matrix $( \Ib - \Lambdab^{k} +\mu \check{\Sigmab }^{-1} - \rho(\Sb^{k}+\Ub^{k} ))/\mu $.
	\item The $\Sb$-subproblem \eqref{subprob_S} can be expressed as
	\begin{equation}
		\begin{aligned}
			\Sb^{k+1} = \underset{\Sb}{\argmin}\ & f(\Lb^{k+1}, \Sb) +C\|\Sb\|_0 -\innerprod{\Thetab^k}{\Sb} \\
			& + \frac{\rho}{2} \|\Sb-\Vb^k\|_\F^2.
		\end{aligned}
	\end{equation}
	Due to the presence of $\|\Sb\|_0$, it is difficult to find a closed-form solution.
 In this case, we employ a proximal gradient descent step
		\begin{equation}\label{solution_S}
			\begin{aligned}
				\Sb^{k+1} =  \prox_{\gamma C \|\cdot\|_0} \left[ \Sb^k - \gamma \left(\nabla_\Sb f(\Lb^{k+1}, \Sb^k)\right.\right.\\
				\qquad \left.\left.-\Thetab^k +\rho(\Sb^k-\Vb^k)\right) \right],
			\end{aligned}
		\end{equation}
	where $\gamma>0$ is a step size.
	\item The $\Ub$-subproblem \eqref{subprob_U} can be reduced to
	\begin{equation}\label{subprob_U_redu}
		\Ub^{k+1}=\underset{\Ub\succeq 0}{\argmin}\ \innerprod{\Lambdab^k}{\Ub}+\frac{\rho}{2}\|\Lb^{k+1}-\Ub\|_\F^2.
	\end{equation}
	In order to deal with the  constraint $\Ub\succeq 0$, we introduce the projection operator $\Pi_{\Scal_+^p}(\cdot)$ onto $\Scal_{+}^p$,
 the cone of $p\times p$ positive semidefinite matrices. Let $\Tb=\sum_{i=1}^p \sigma_i \tb_i \tb_i^\top$ be the spectral decomposition of a symmetric matrix $\Tb$, where $\sigma_i$ and $\tb_i$ represent the $i$-th eigenvalue and the corresponding eigenvector. Then the projection operator can be expressed as
	\begin{equation}
	 	\Pi_{\Scal_+^p}(\Tb)=\sum_{i=1}^{p} \max\{\sigma_i,0\} \tb_i \tb_i^\top.
	\end{equation}
With the help of the projection operator, we take
		\begin{equation}\label{solution_U}
		\Ub^{k+1}=\Pi_{\Scal_{+}^p}(\Lb^{k+1}-\tfrac{1}{\rho} \Lambdab^{k})
	\end{equation}
 as a solution to \eqref{subprob_U_redu}.
	\item The $\Vb$-subproblem \eqref{subprob_V} is similar to $\eqref{subprob_U}$, and a minimizer  can be calculated by
	\begin{equation}\label{solution_V}
		\Vb^{k+1}=\Pi_{\Scal_{+}^p}(\Sb^{k+1}-\tfrac{1}{\rho} \Thetab^{k}).
	\end{equation}
\end{itemize}
	
Algorithm \ref{alg:admm} summarizes the steps of the ADMM to solve \eqref{optim_model_admm}.
\begin{algorithm}
	\caption{$\ell_0$ ADMM for \eqref{optim_model_2}}
	\label{alg:admm}
	\begin{algorithmic}[1]
		\Require $C$, $\mu$, $\rho$, $\gamma$, $\check{\Sigmab}$, $\texttt{iter}$.
		\State Set $k=0$, and initialize $(\Lb^0, \Sb^0, \Ub^0, \Vb^0; \Lambdab^0, \Thetab^0)$.
		\While{the stopping condition is not met and $k\le \texttt{iter}$}
		\State Update $\Lb^{k+1}$ by \eqref{solution_L}.
		\State Update $\Sb^{k+1}$ by \eqref{solution_S}.
		\State Update $\Ub^{k+1}$ by \eqref{solution_U}.
		\State Update $\Vb^{k+1}$ by \eqref{solution_V}.
		\State Update $\Lambdab^{k+1}$ by \eqref{dual_update_Lambda}.
		\State Update $\Thetab^{k+1}$ by \eqref{dual_update_Theta}.
		\State Set $k=k+1$.
		\EndWhile \\
		\Return the final iterate $(\Lb^k, \Sb^k, \Ub^k, \Vb^k; \Lambdab^k, \Thetab^k)$.
 	\end{algorithmic}
\end{algorithm}

\subsubsection{Initialization of Algorithm~\ref{alg:admm}}
We first compute the spectral decomposition of $\check{\Sigmab}=\Qb\Jb\Qb^\top$ where $\Jb=\diag(\tau_1, \tau_2, \cdots, \tau_p)$ is the eigenvalue matrix such that $\tau_1\geq \tau_2\geq \cdots \geq \tau_p$. Let $\tilde{\Jb}:=\diag(\tau_1, \tau_2, \cdots, \tau_d,0,\cdots,0)$ with $d<p$. Then we initialize the $\ell_0$ ADMM as follows:
 \begin{equation}\label{initial}
	\Lb^0=\Qb\tilde{\Jb}\Qb^\top, \ \Sb^0=\check{\Sigmab}-\Lb^0, \ \Ub^0=\Lb^0,\ \Vb^0=\Sb^0,
\end{equation}
 and the initial Lagrange multiplier matrices $\Lambdab^0$ and $\Thetab^0$ are set to zero matrices.

\subsubsection{Stopping condition of Algorithm~\ref{alg:admm}}
We terminate Algorithm \ref{alg:admm} when two successive iterates are sufficiently close.  Such a condition  can be described as 
\begin{equation}\label{upperbound_error}
	\max\{\beta_1, \beta_2, \beta_3, \beta_4, \beta_5, \beta_6\}<\texttt{tol},
\end{equation} 
where  $\texttt{tol}>0$ represents the tolerance level, and
\begin{equation}\label{termi_cond}
	\begin{aligned}
		\beta_1 & := \|\Lb^{k+1} -\Lb^k\|_\F ,\ \beta_2 :=\|\Sb^{k+1} -\Sb^k\|_\F,\\
		\beta_3& := \|\Ub^{k+1} -\Ub^k\|_\F,\ \beta_4 :=\|\Vb^{k+1} -\Vb^k\|_\F, \\
		\beta_5& := \|\Lambdab^{k+1} -\Lambdab^k\|_\F,\ \beta_6 := \|\Thetab^{k+1} -\Thetab^k\|_\F.
	\end{aligned}
\end{equation}

\subsection{Convergence analysis}
In this subsection, we present convergence analysis of the proposed ADMM algorithm. 
For convenience, let us set $\Ab^k:=(\Lb^k,\Sb^k,\Ub^k,\Vb^k;\Lambdab^k,\Thetab^k)$.
Suppose that $\{\Ab^k\}$ is a sequence generated by Algorithm \ref{alg:admm}. 
The next lemma shows that the sequence of function values $\{\Lcal_\rho(\Ab^k)\}$ satisfies a kind of \emph{sufficient descent} condition at each iteration.

\begin{lemma}\label{Sufficient decrease}
	For a sufficiently large penalty parameter $\rho$, we choose any
		$\gamma\leq \frac{1}{K+2\rho}$. Then the ADMM iterates $\{\Ab^k\}$ satisfy the inequality
	\begin{equation}\label{decrease_alf}
		\begin{aligned}
		&\Lcal_\rho(\Ab^{k})-\Lcal_\rho(\Ab^{k+1}) \\
		\geq & \frac{\epsilon}{2}\left(\|\Lb^{k+1}-\Lb^k\|_\F^2+\|\Sb^{k+1}-\Sb^k\|_\F^2  +\|\Ub^{k+1}-\Ub^k\|_\F^2 \right. \\
		& \left.+\|\Vb^{k+1}-\Vb^k\|_\F^2\right) +\frac{1}{\rho} \left(\|\Lambdab^{k+1}-\Lambdab^k\|_\F^2+\|\Thetab^{k+1}-\Thetab^k\|_\F^2 \right)
	\end{aligned}
	\end{equation}
	for some sufficiently small $\epsilon>0$ and all $k\geq 1$. 
\end{lemma}

\begin{proof}
	See the appendix.
\end{proof}
Based on the above intermediate result, next we demonstrate the boundedness property of the whole sequence of the ADMM iterates.
\begin{proposition}\label{sequence_bound}
Under the conditions of Lemma \ref{Sufficient decrease}, suppose further that the sequence of dual variables $\{\Lambdab^k,\Thetab^k\}$ is bounded. Then $\{\Ab^k\}$ is bounded.
\end{proposition}
\begin{proof}
 Using the argument in the proof of Lemma \ref{lem_lower_semiconti},
we can show that the augmented Lagrangian $\Lcal_\rho$ is lower semicontinuous. 
Let $\Ab^0$ be the initialization of Algorithm~\ref{alg:admm}.
According to Lemma \ref{Sufficient decrease}, for $k\geq 1$ we have
	\begin{equation}\label{L_rho_upperbound}
		\begin{aligned}
			&\Lcal_\rho(\Ab^0)>\Lcal_\rho(\Ab^k)  \\ &=f(\Lb^k,\Sb^k) +C\|\Sb^k\|_0+\frac{\rho}{2} \|\Lb^k-\Ub^k\|_\F^2 + \frac{\rho}{2} \|\Sb^k-\Vb^k\|_\F^2 \\
			 &\quad -\innerprod{\Lambdab^k}{\Lb^k-\Ub^k} 
			-\innerprod{\Thetab^k}{\Sb^k-\Vb^k}.
		\end{aligned}
\end{equation}
Exploiting the boundedness assumption on $\{\Lambdab^k,\Thetab^k\}$, we can show without difficulty that the quadratic function $\frac{\rho}{2} \|\Lb^k-\Ub^k\|_\F^2-\innerprod{\Lambdab^k}{\Lb^k-\Ub^k}$ is bounded from below, and that the function value diverges to infinity if $\|\Lb^k\|_\F$ or $\|\Ub^k\|_\F$ tends to infinity separately. The same argument applies to the quadratic function $\frac{\rho}{2} \|\Sb^k-\Vb^k\|_\F^2-\innerprod{\Thetab^k}{\Sb^k-\Vb^k}$.
	
Now we are ready to prove the boundedness of $\{\Ab^k\}$. 
Suppose first that $\{\Lb^k\}$ is unbounded, which means that there exists a subsequence, still written as $\Lb^k$, such that
$\|\Lb^k\|_\F\to\infty$. Then we have
	$$\Lcal_\rho(\Ab^k) \geq f(\Lb^k,\Sb^k) +C\|\Sb^k\|_0+\mathrm{const.}.$$
It follows from Lemma \ref{lem_diverg} that $\Lcal_\rho(\Ab^k)$ approaches infinity, which is a contradiction to \eqref{L_rho_upperbound}. Therefore, $\{\Lb^k\}$ remains bounded. In the same manner, we can derive the boundedness of $\{\Sb^k\}$. In addition, boundedness of $\{\Ub^k\}$ and $\{\Vb^k\}$ can be shown via properties of the quadratic functions which are observed right after \eqref{L_rho_upperbound}.
 Consequently, the boundedness of $\{\Ab^k\}$ is established.
\end{proof}

By the Bolzano--Weierstrass theorem, Proposition~\ref{sequence_bound} implies that there exists a  subsequence $\{\Ab^{k_j}\}$ that converges to a cluster point $\Ab^*$. The next result shows that two successive iterates are eventually close to each other.

\begin{proposition}
    Under the conditions of Lemma \ref{Sufficient decrease}, we have $\lim_{k\to\infty}\|\Ab^{k+1}-\Ab^k\|_\F^2=0$.
\end{proposition}

\begin{proof}
According to the proof of Lemma~\ref{lem_diverg}, the objective function in \eqref{optim_model_2} is bounded from below. It is then straightforward to show that the same holds for the augmented Lagrangian \eqref{augment_function}, i.e., $\Lcal_\rho(\Ab)\geq C_1$ for some constant $C_1$.
By Lemma \ref{Sufficient decrease}, the sequence $\Lcal_\rho(\Ab^k)$ of function values is monotonically decreasing, and hence must converge by the monotone convergence theorem.
More precisely, we have
    \begin{equation}
		\begin{aligned}
			\infty&>\Lcal_\rho(\Ab^1)-C_1 \\
			&\geq\Lcal_\rho(\Ab^1)-\Lcal_\rho(\Ab^{N+1}) \\
			&=\sum_{k = 1}^N\left[\Lcal_\rho(\Ab^{k})-\Lcal_\rho(\Ab^{k+1}) \right] \\
			&\geq \frac{\epsilon}{2} \left(\sum_{k = 1}^N \|\Lb^{k+1}-\Lb^k\|_\F^2+\sum_{k = 1}^N\|\Sb^{k+1}-\Sb^k\|_\F^2  \right.\\
			&\quad \left.+\sum_{k = 1}^N\|\Ub^{k+1}-\Ub^k\|_\F^2+\sum_{k = 1}^N\|\Vb^{k+1}-\Vb^k\|_\F^2 \right)\\
			&\quad +\frac{1}{\rho} \left( \sum_{k = 1}^N\|\Lambdab^{k+1}-\Lambdab^k\|_\F^2+\sum_{k = 1}^N\|\Thetab^{k+1}-\Thetab^k\|_\F^2 \right).
		\end{aligned}
    \end{equation}
	Letting $N\rightarrow\infty$, we have
	\begin{equation}\label{variable_difference}
		\begin{aligned}
			\sum_{k=1}^{\infty}\left\|\Lb^{k+1}-\Lb^k\right\|_\F^2<\infty, \ \sum_{k=1}^{\infty}\left\|\Sb^{k+1}-\Sb^k\right\|_\F^2<\infty, \\
			\sum_{k=1}^{\infty}\left\|\Ub^{k+1}-\Ub^k\right\|_\F^2<\infty, \ \sum_{k=1}^{\infty}\left\|\Vb^{k+1}-\Vb^k\right\|_\F^2<\infty, \\
			\sum_{k=1}^{\infty}\left\|\Lambdab^{k+1}-\Lambdab^k\right\|_\F^2<\infty,\ \text{and}\ \sum_{k=1}^{\infty}\left\|\Thetab^{k+1}-\Thetab^k\right\|_\F^2<\infty,
		\end{aligned}
	\end{equation}
 which clearly implies
	$\lim_{k\to\infty}\|\Ab^{k+1}-\Ab^k\|_\F^2=0$.
\end{proof}

\begin{remark}
	Global convergence of the ADMM is hard to establish without the \emph{Kurdyka--\L ojasiewicz (K\L) property} in the presence of a nonconvex or nonsmooth term in the objective function \cite{attouch2013convergence,bolte2014proximal, yang2017alternating}.
	It is known that \emph{semialgebraic} functions and \emph{continuous subanalytic} functions over closed domains satisfy the K\L\ property. While our objective function can be written as a sum of functions of these two types, 
	 it is in general not true that the K\L\ functions have the additive property. For this reason, we are unable to prove that our augmented Lagrangian \eqref{augment_function} satisfies the K\L\ property.  Therefore, the issue of global convergence appears very challenging and will be left for future research.
\end{remark}

In the final part of this subsection, we show that \emph{if} the whole sequence of the ADMM iterates converges, the cluster point is indeed a locally optimal solution to \eqref{optim_model_2}.

\begin{proposition}
		Assume that $\{\Ab^k\}$, the sequence generated by Algorithm \ref{alg:admm}, converges and the cluster point $\Ab^*=(\Lb^*, \Sb^*,\Ub^*,\Vb^*;\Lambdab^*,\Thetab^*)$ is still feasible for \eqref{optim_model_2}. Then $\Ab^*$ is a P-stationary point of \eqref{optim_model_2} and hence a local minimizer.
\end{proposition}

\begin{proof}
 Taking the limit of \eqref{dual_update_Lambda}, we have $\Lambdab^{*}  = \Lambdab^* - \rho(\Lb^{*}-\Ub^{*})$ and thus 
	\begin{equation}\label{lstar}
	   	\Lb^{*}=\Ub^{*}.
	\end{equation}
	Since the subproblem \eqref{subprob_U} for $\Ub$ is smooth, we can write down its  KKT conditions as
	\begin{subequations}\label{KKT_subprob_U}
		\begin{align}
			\Ub^{k+1} & \succeq 0 \\ 
			\Yb^{k+1} & \succeq 0 \\ 
			\trace(\Ub^{k+1} \Yb^{k+1}) & = 0 \\ 
			\Lambdab^k - \rho(\Lb^{k+1}-\Ub^{k+1})-\Yb^{k+1} & = O,  \label{relat_Lambda_Y}
		\end{align}
	\end{subequations}
	where $\Yb$ is the Lagrange multiplier for $\Ub\succeq 0$. It follows from \eqref{relat_Lambda_Y} that $\Yb^{k+1}=\Lambdab^k - \rho(\Lb^{k+1}-\Ub^{k+1})$ must also converge to some $\Yb^*$.
 In view of \eqref{lstar}, we immediately have $\Lambdab^*=\Yb^* $. 
 Now we let $k\to\infty$, and  the relations in \eqref{KKT_subprob_U} can be rewritten as 
 $\Lb^{*}\succeq 0$,  $\Lambdab^{*} \succeq 0$, and $\trace(\Lambdab^{*} \Lb^{*}) = 0$.

 Similarly, a limit argument on \eqref{dual_update_Theta} yields $\Sb^*=\Vb^*$, and a KKT analysis on \eqref{subprob_V} gives the conditions
$\Sb^{*} \succeq 0$, $\Thetab^{*} \succeq 0$, and $\trace(\Thetab^{*}\Sb^{*}) = 0$. 

	The first-order optimality condition of the subproblem \eqref{subprob_L} is $\nabla_\Lb f(\Lb^{k+1},\Sb^k) -\Lambdab^{k+1} + \rho(\Lb^{k+1}-\Ub^k) = O$. Taking the limit as $k\to\infty$, we arrive at \eqref{P-stat-Lagrangian-grad-1_v2}.
	In a similar fashion, if we take limits on both sides of \eqref{solution_S}, we obtain \eqref{P-stat-Lagrangian-prox_v2}.
	The remaining condition \eqref{P-stat-primal-2_v2} holds by the premise of theorem. Thus we have shown that the cluster point $\Ab^*$
 is a P-stationary point, and its local optimality is a consequence of Theorem~\ref{thm_optimality}. 
\end{proof}

\section{Numerical Experiments}\label{sec:numerical}

In this section, we evaluate the performance of Algorithm~\ref{alg:admm} via numerical experiments on synthetic and real economic datasets. All these simulations are implemented in MATLAB (R2022a) on a MacBook Pro with 64 GB RAM and an Apple M1 Max processor.

\subsection{Synthetic data examples}

In this subsection, we conduct Monte Carlo simulations, each having 100 repeated trials, following the steps outlined below.
\begin{enumerate}
\item[i)] 
Randomly generate a matrix $\Gammab\in\Rbb^{p\times r}$ with rank $r$ and a \lyn{sparse} $\hat\Sb\succ 0$ of size $p$, where the cross-sectional dimension $p=40$.
\item[ii)] 
Generate a sample sequence $\{\yb_i\}$ of length $N$ according to the FA model \eqref{generate_model} with $\mub=\zerob$.
\item[iii)]  
Compute the sample covariance matrix
$$\check{\Sigmab} = \frac{1}{N} \sum_{i=1}^{N} \yb_i \yb_i^T.$$
\item[iv)]  
 Select parameters $(C, \mu, \rho)$ through the Cross Validation (CV) and then run Algorithm \ref{alg:admm} to compute the estimate $(\Lb^*, \Sb^*)$.
 \item[v)]  Compute the numerical rank $r^*$ of the optimal  $\Lb^*$ via
 \begin{equation}\label{factor_estimate}
 	r^*:=\underset{i\leq i_{\max}}{\argmax}\ {\lambda_i}/{\lambda_{i+1}},
 \end{equation}	
 where $\lambda_1\geq\lambda_2\geq\cdots\geq\lambda_p$ represent the eigenvalues of $\Lb^{*}$, and $i_{\max}$ denotes the minimum index $i$ satisfying $\lambda_{i+1}/\lambda_{i}<0.05$  in accordance with \cite{ciccone2018factor}.
\end{enumerate}

\subsubsection{Evaluation metrics}
Estimation of the number of factors is an important task in FA. Two performance indices are proposed in \cite{ciccone2018factor} which are also used here:
\begin{itemize}
	\item  the root mean squared error (RMSE) between the true rank $r$ and the numerical rank $r^*$, namely
		\begin{equation}\label{rmse}
				\mathrm{RMSE} = \sqrt{\frac{1}{100}\sum_{t=1}^{100}(r_t^*-r)^{2}},
		\end{equation}
  where $t$ is a label for trials in a Monte Carlo simulation, 
	\item \lyn{the subspace recovery accuracy}
	\begin{equation}\label{ratio}
		\text{ratio}(\mathbf{\Gamma}^{*}):=\operatorname{tr}(\mathbf{\Gamma}^\top\Pb\mathbf{\Gamma})/\operatorname{tr}(\mathbf{\Gamma}^\top\mathbf{\Gamma}),
	\end{equation}
	where $\Gammab$ and $\Gammab^{*}$ are the true and estimated factor loading matrices, respectively, and $\Pb$ denotes the projection matrix onto the column space (range) of $\Gammab^{*}$. 
	Evidently, the above ratio is between $0$ and $1$, and a larger value indicates better alignment of the column spaces of $\Gammab$ and $\Gammab^*$. 
\end{itemize}

\subsubsection{Parameters setting} \label{subsubsec:para_setting}

We set the  tolerance parameter in \eqref{upperbound_error} to be  $\texttt{tol}=10^{-3}$ and the maximum number of iterations $\texttt{iter}=10^4$. The tuning parameters $\mu$ and $C$ represent the trade-off between the low-rankness of $\Lb^*$, the model misfit as measured by the KL divergence, and the sparsity of $\Sb^*$, while $\rho$ is the penalty parameter in the augmented Lagrangian \eqref{augment_function}.
We use a CV procedure to select them. 
The candidate set for $C$ and $\mu$ is \lyn{ $\left\{60,110,\cdots,360\right\}$}
 and  $\rho$ is selected from $\left\{2^0, 2^1, \cdots, 2^5\right\}$. We randomly partition the samples $\yb_i$ generated in Step ii) above into training and validation sets of equal size. Using the optimal solution $(\tilde{\Lb}, \tilde{\Sb})$ computed on the training set under a specific parameter configuration $(C, \mu, \rho)$ and the sample covariance matrix $\check{\Sigmab }^{\vrm}$ computed on the validation set, we define the following score function
\begin{equation}\label{score_func}
	(r^*_{\tilde{\Lb}}+\|\tilde{\Sb}\|_{0}) \Dcal_{\KL} (\tilde{\Lb}+\tilde{\Sb} \| \check{\Sigmab}^{\vrm} ).
\end{equation}
The parameter triple $(C,\mu,\rho)$ corresponding to the minimum score function value is selected.

Next, we consider the selection of the parameter $\gamma$ in the update of $\Sb$ in \eqref{solution_S} which comes from the proximal operator \eqref{prox_scalar_sol1} and can be interpreted as a step size.
We must emphasize that in our implementation, once $\gamma$ is selected, it does not change in each simulation, including the CV and the final solution with the chosen parameters $(C,\mu,\rho)$.
If  $\gamma$ is large, it will immediately set $\Sb$ into a zero matrix due to the thresholding operation, thus leading to convergence of the $\Sb$ update.
On the other hand, when $\gamma$ is extremely small, the $\ell_0$ proximal operator \eqref{prox_scalar_sol1} essentially reduces to an identical mapping, and Algorithm \ref{alg:admm} is simply the ADMM for \eqref{optim_model_2} without the sparsity enforcing term $C\|\Sb\|_0$. As a consequence, such a solution does not induce sparsity on $\Sb$. This point can also be observed from our simulations.
Indeed,  we take $\gamma=10^{-2}, 10^{-4}, 10^{-6}, 10^{-8}$, respectively, and perform Monte Carlo simulations on the same dataset with \lyn{$N=1000$, the true rank $r = 4$, and a signal-to-noise ratio (SNR, defined as $\|\Gammab\Gammab^\top\|_\F/\|\hat\Sb\|_\F$) equal to $6$}. To speed up the convergence of Algorithm~\ref{alg:admm}, a projection step onto the cone of semidefinite matrices was incorporated into the updates of $\Lb$. Fig.~\ref{fig:gamma} shows the results in terms of three metrics:
the KL divergence between the computed $\Lb^*+\Sb^*$ and the \emph{true} covariance matrix $\Sigmab$, $\|\Sb^*\|_0$, and the running time. The subfigure~\ref{fig:gamma}(ii)  indicates that when $\gamma=10^{-2}$, the returned $\Sb^*$ is a zero matrix. In contrast, when $\gamma=10^{-6}$ or $\gamma=10^{-8}$,  $\Sb^*$ has poor sparsity. The choice $\gamma=10^{-4}$ appears to produce the best estimate $\Sb^*$. 
In addition, Table~\ref{table:gamma} presents the RMSE \eqref{rmse} of the numerical rank estimate \eqref{factor_estimate} for two Monte Carlo simulations in which $N = \lyn{1000}$ and the true rank is $r = \lyn{4}$ and $r = 10$, respectively. 
In view of the results in Fig.~\ref{fig:gamma} and Table~\ref{table:gamma}, \lyn{the value $\gamma=10^{-4}$ seems reasonable for synthetic datasets with $\SNR=6$}.

\begin{table}[h]
	\centering
	\caption{RMSE of the numerical rank of $\Lb^*$}
	\lyn{
	\begin{tabular}{ccccc}
		\toprule
		\quad\quad  & $\gamma=10^{-2}$ & $\gamma=10^{-4}$ & $\gamma=10^{-6}$ & $\gamma=10^{-8}$  \\
		\midrule
		$r=4$ & 0  & 0 & 0 & 1.3342 \\
		$r=10$ & 0.1000  &0 & 0.1000 & 2.8513 \\
		\bottomrule
	\end{tabular}
}
	\label{table:gamma}
\end{table}
\begin{figure}[t]
	\begin{minipage}[b]{0.3\linewidth}
		\centering
		\centerline{\includegraphics[width=3.3cm]{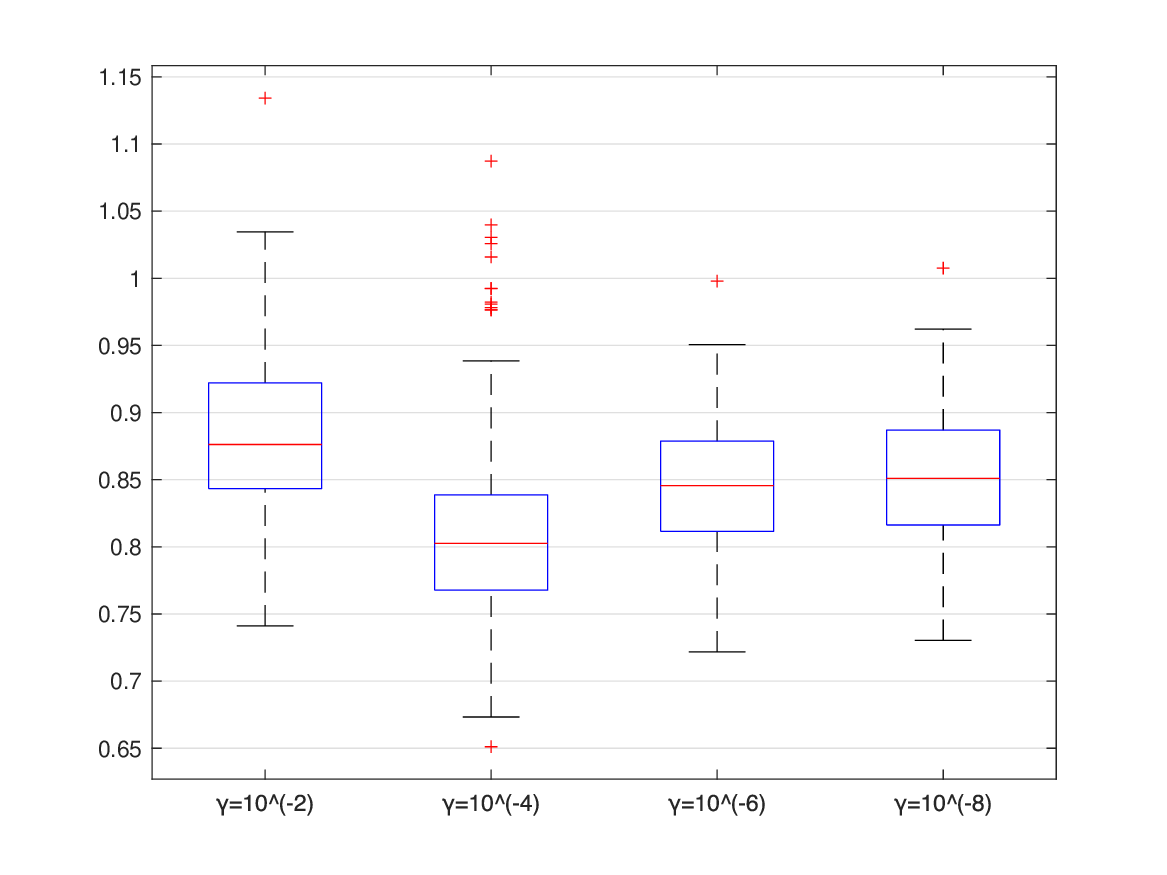}}
		\centerline{\small (i) $\Dcal_{\KL} (\Lb^*+\Sb^* \| \Sigmab$)}
	\end{minipage}
	\hfill
	\begin{minipage}[b]{0.3\linewidth}
		\centering
		\centerline{\includegraphics[width=3.3cm]{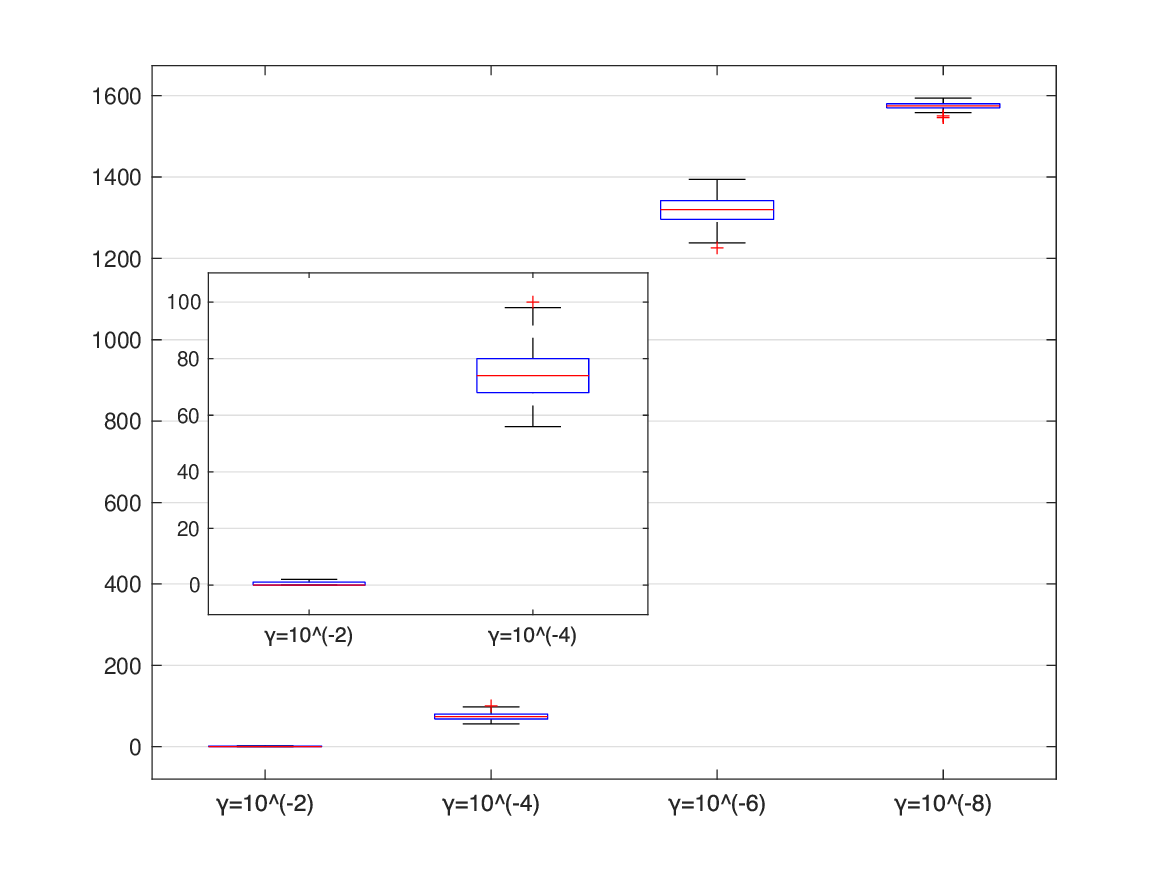}}
		\centerline{\small (ii) $\|\Sb^*\|_0$ }
	\end{minipage}
    \hfill
    \begin{minipage}[b]{.3\linewidth}
		\centering
		\centerline{\includegraphics[width=3.3cm]{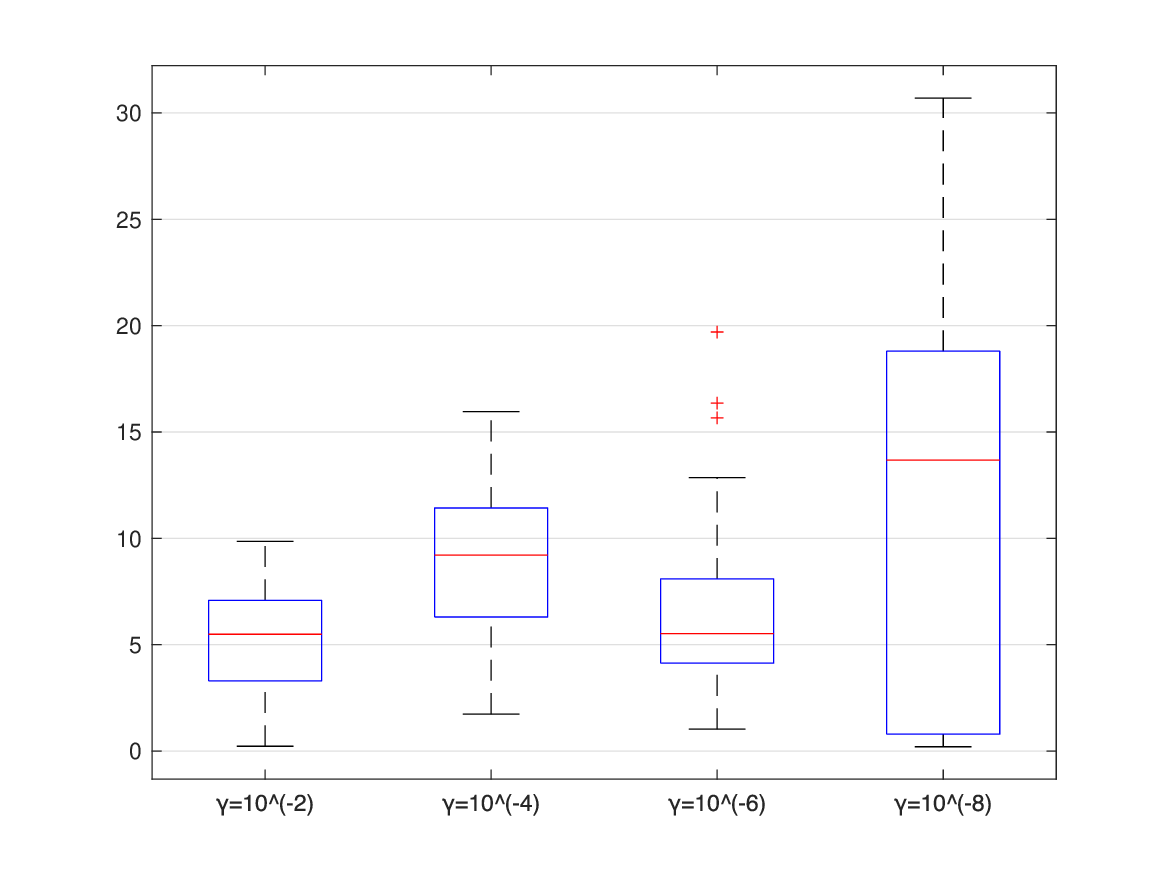}}
		\centerline{\small (iii) Runtime}
	\end{minipage}
	\lyn{
	\caption{The performance of the $\ell_0$ ADMM in terms of three metrics when $\gamma$ takes the value of $10^{-2}$, $10^{-4}$, $10^{-6}$, and $10^{-8}$, respectively, with $N=1000$, the true $r=4$, $\mathrm{SNR}=6$ and the sparsity level of $\hat\Sb$ being $\|\hat \Sb\|_0/p^2=0.055$.}
	\label{fig:gamma}
	}

\end{figure}

\subsubsection{Comparison with the $\ell_1$ ADMM} 

In order to examine the sparsification effect of the $\ell_0$ norm on $\Sb$ via a comparison, we consider the convex relaxation model of \eqref{optim_model_2}, where the sparsity enforcing term is replaced by $C\|\Sb\|_1:=\sum_{i, j} C|s_{ij}|$.  The resulting problem can also be solved via the ADMM, called $\ell_1$ ADMM, once we replace the $\ell_0$ proximal operator in Algorithm~\ref{alg:admm} with the corresponding $\ell_1$ version whose scalar form reads as
\begin{equation}\label{prox_l1}
	\prox_{\gamma C |\cdot|_1} (s)=\sign(s)\cdot \max\{|s|-\gamma C, 0\},
\end{equation}
and the matrix form is also an elementwise application of the scalar form above. The $\ell_1$ proximal operator is also called \emph{soft thresholding operator} in the literature. \lyn{Such a comparison is also motivated by  the fact that the globally optimal solution to the $\ell_1$ relaxed optimization problem is \emph{not} even locally optimal in the $\ell_0$ sense, see \cite[Theorem~1]{FA_CDC_24} for the precise statement}.

\lyn{We separately consider two structures of the true $\hat\Sb$ under $\mathrm{SNR}=6$: diagonal and nondiagonal sparse.
When $\hat \Sb$ is set as a diagonal matrix, more precisely, a scaled identity matrix of the form $\alpha \Ib$ with $\alpha>0$, 
panel (a) of Fig.~\ref{fig:0_1_compare_total} shows graphically the estimates $\Sb^*$ computed from the $\ell_0$ and the $\ell_1$ versions of the ADMM at different values of $\gamma$. 
In the case of $\gamma=10^{-4}$,} we note that the $\ell_0$ ADMM correctly recognizes the underlying diagonal structure of $\hat{\Sb}$, while the $\ell_1$ version returns a zero matrix which is similar to the outcome of the $\ell_0$~ADMM for $\gamma = 10^{-2}$. See the comment in Subsubsec.~\ref{subsubsec:para_setting}. 
 When $\gamma=10^{-6}$, the $\ell_1$ ADMM returns only a few nonzero entries along the diagonal of $\Sb^*$ with small numerical values and fails to reveal the diagonal structure of $\hat \Sb$.  For the case of $\gamma=10^{-8}$, the $\ell_1$ ADMM treats the optimization problem that essentially neglects the sparsity enforcing term $C\|\Sb\|_1$, much similar to the $\ell_0$ case. 
 Indeed, both the $\ell_0$ and the $\ell_1$ proximal operators reduce to the identity mapping when $\gamma\to 0$.
The simulation results are also not far from each other: the $\ell_1$ estimate of $\Sb^*$ has many spurious off-diagonal nonzeros and the situation of the $\ell_0$ estimate can be seen from panel (ii) of Fig.~\ref{fig:gamma}. 
\lyn {
Moreover, when $\hat \Sb$ is set as a nondiagonal sparse matrix, panel (b) of Fig.~\ref{fig:0_1_compare_total} clearly demonstrates the superior performance of the $\ell_0$ ADMM in 
estimating the true sparse structure.}
In summary, the $\ell_0$ ADMM with a suitable choice of $\gamma$ appears to be able to identify the sparse structure of $\hat \Sb$ in a more reliable fashion.
\begin{figure*}
	\begin{minipage}[b]{0.19 \linewidth}
		\centering
		\centerline{\includegraphics[width=0.95\linewidth]{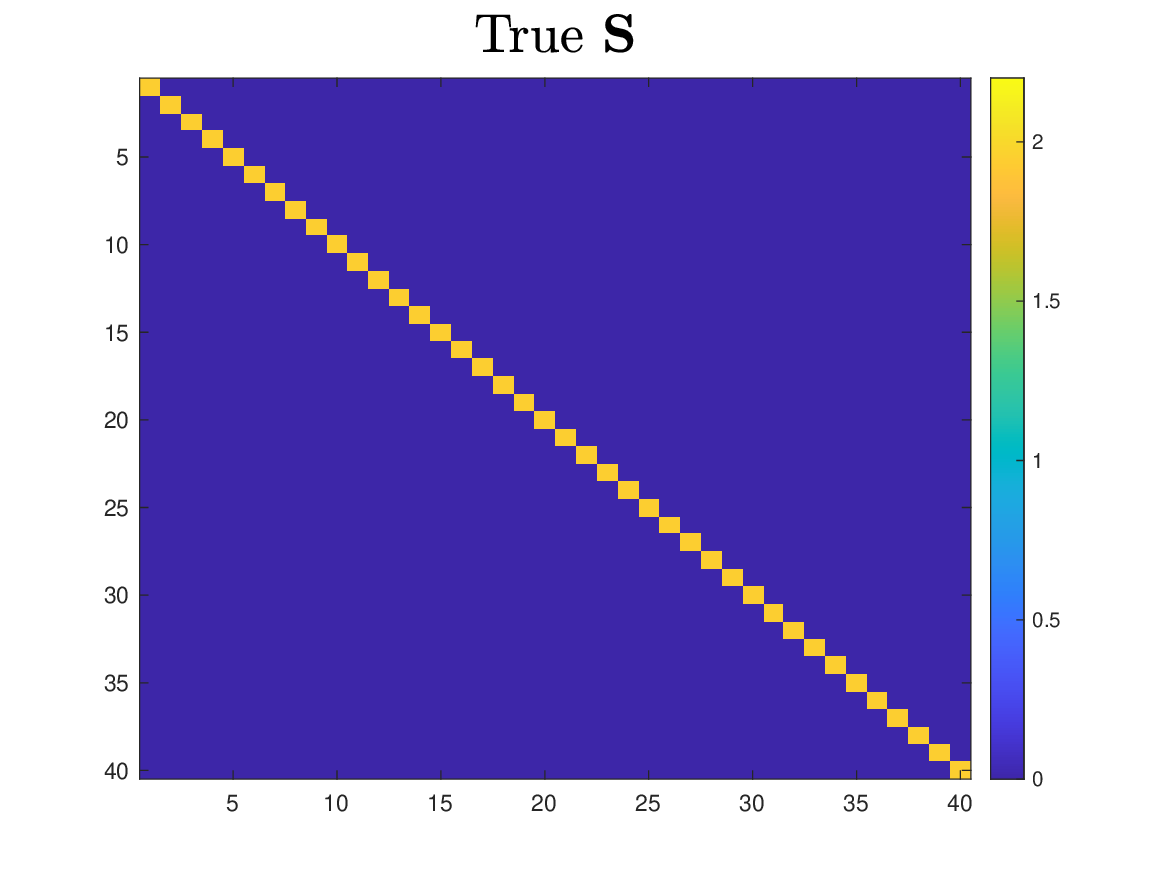}}
	\end{minipage}
	\hfill
	\begin{minipage}[b]{0.19\linewidth}
		\centering
		\centerline{\includegraphics[width=0.95\linewidth]{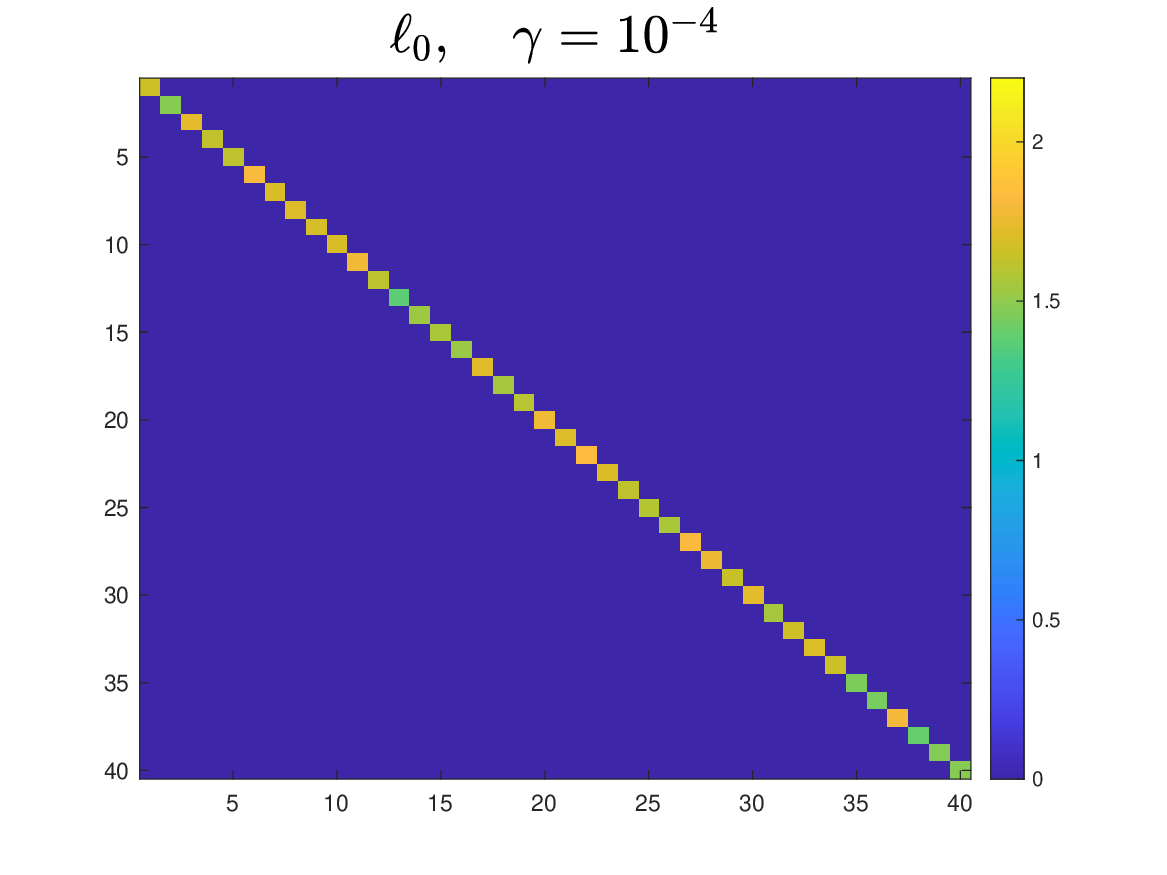}}
	\end{minipage}
	\hfill
	\begin{minipage}[b]{0.19\linewidth}
		\centering
		\centerline{\includegraphics[width=0.95\linewidth]{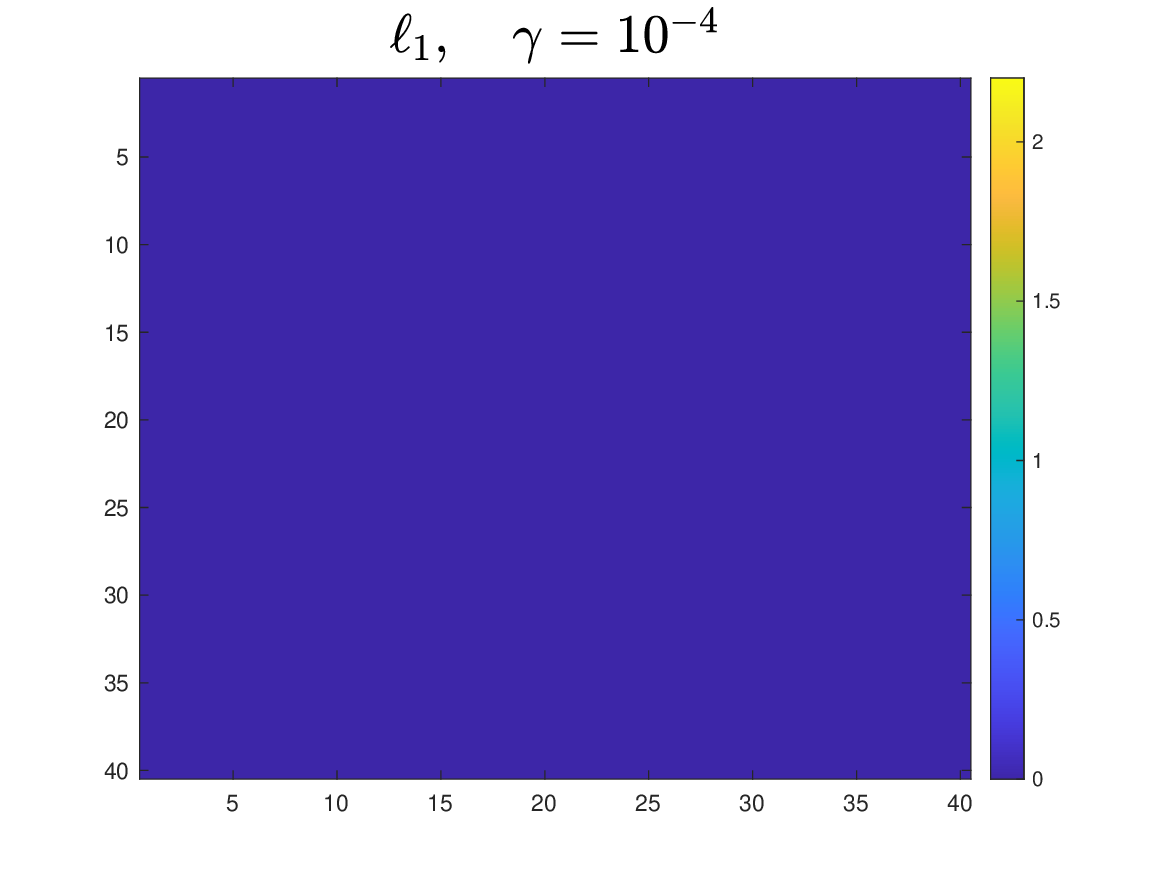}}
	\end{minipage}
	\hfill
	\begin{minipage}[b]{0.19\linewidth}
		\centering
		\centerline{\includegraphics[width=0.95\linewidth]{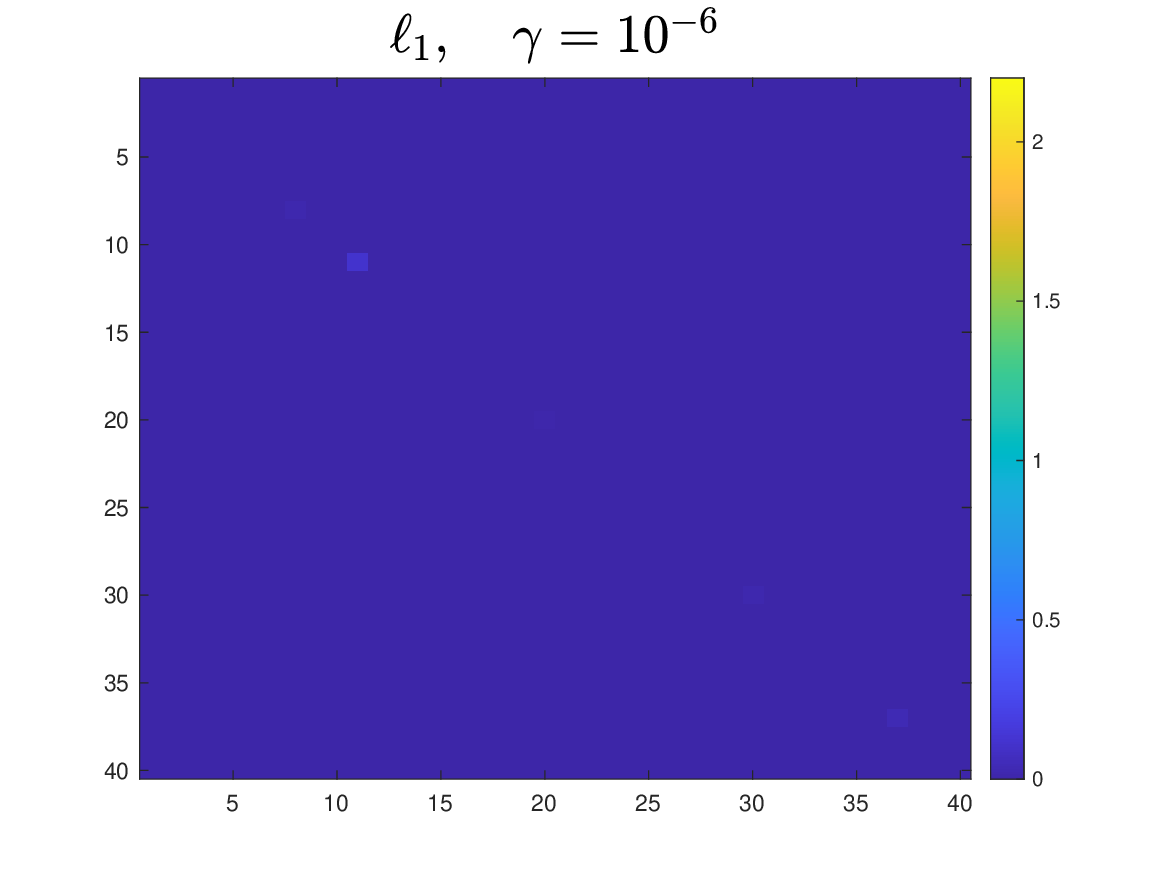}}
	\end{minipage}
	\hfill
	\begin{minipage}[b]{0.19\linewidth}
		\centering
		\centerline{\includegraphics[width=0.95\linewidth]{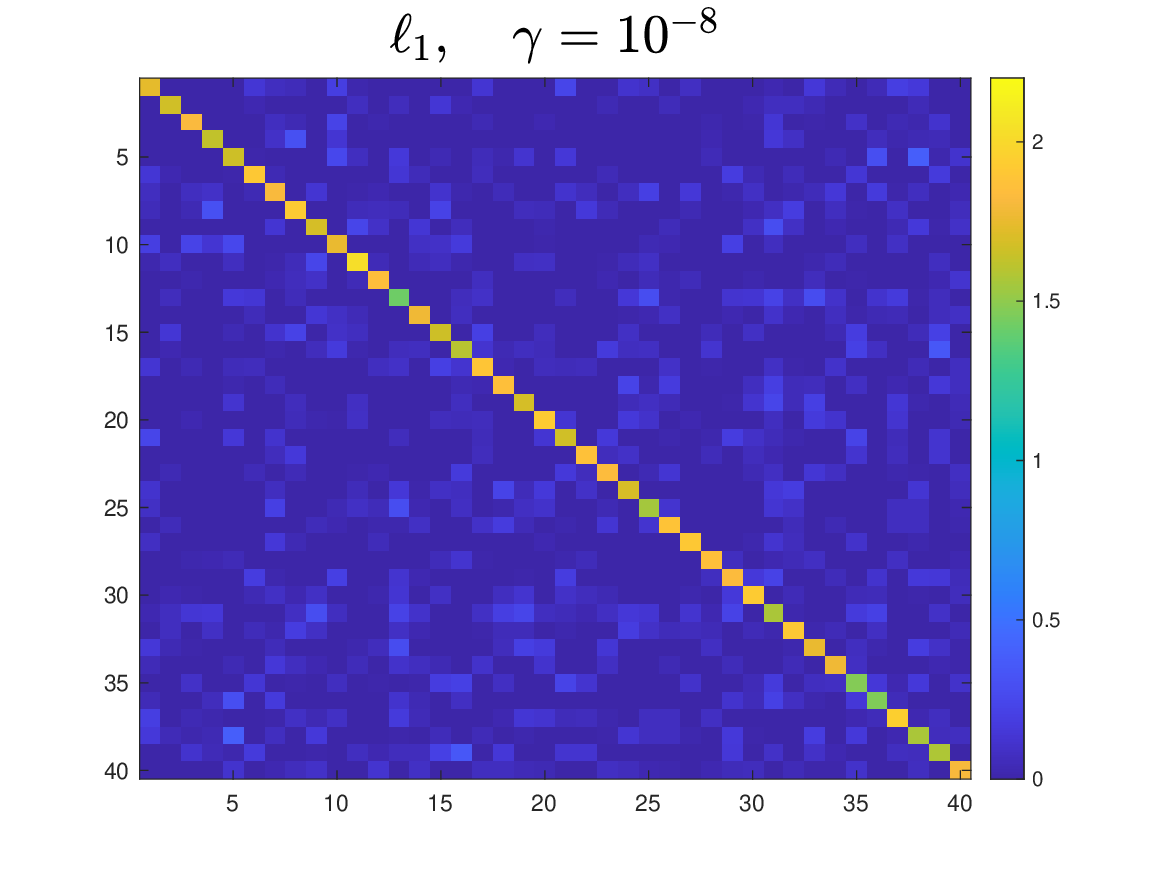}}
	\end{minipage}
	\centerline{\small (a)}
	
	\vspace{0.1cm}
	\begin{minipage}[b]{0.19 \linewidth}
		\centering
		\centerline{\includegraphics[width=0.95\linewidth]{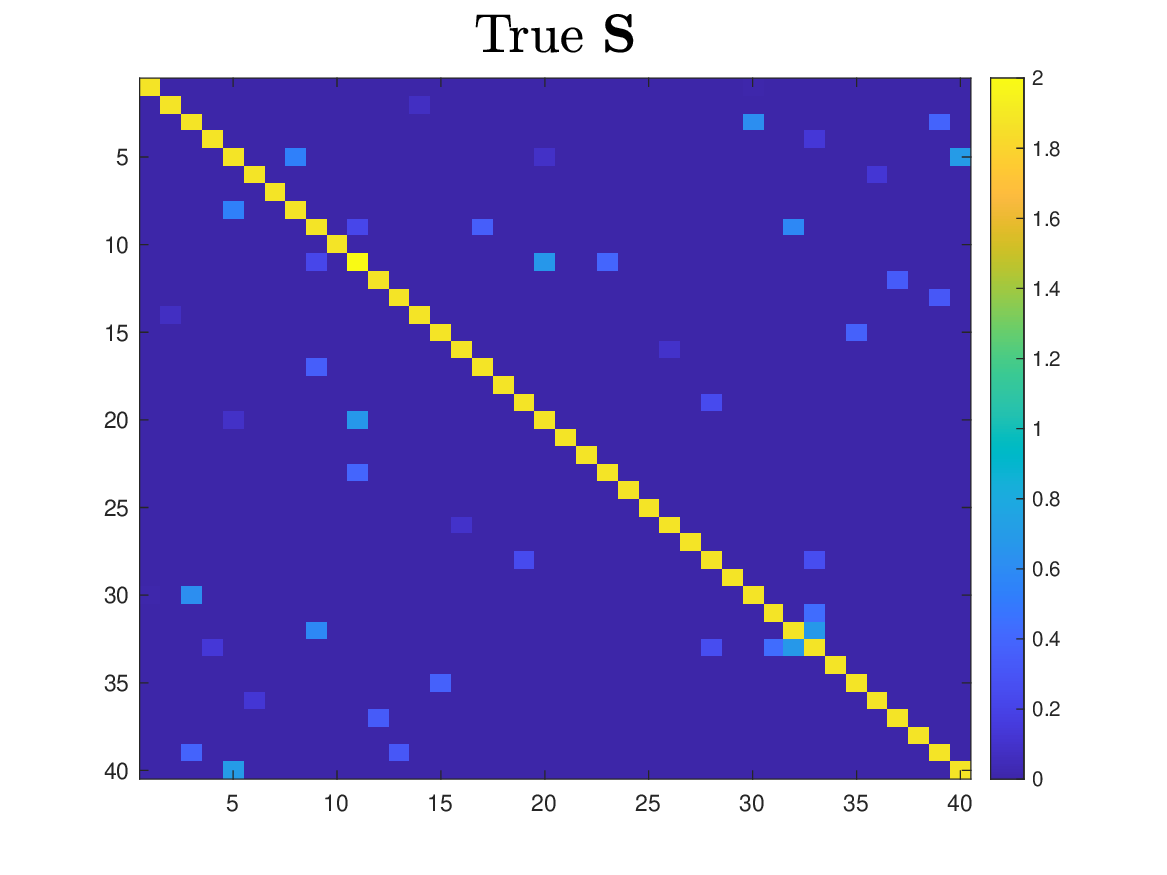}}
	\end{minipage}
	\hfill
	\begin{minipage}[b]{0.19 \linewidth}
		\centering
		\centerline{\includegraphics[width=0.95\linewidth]{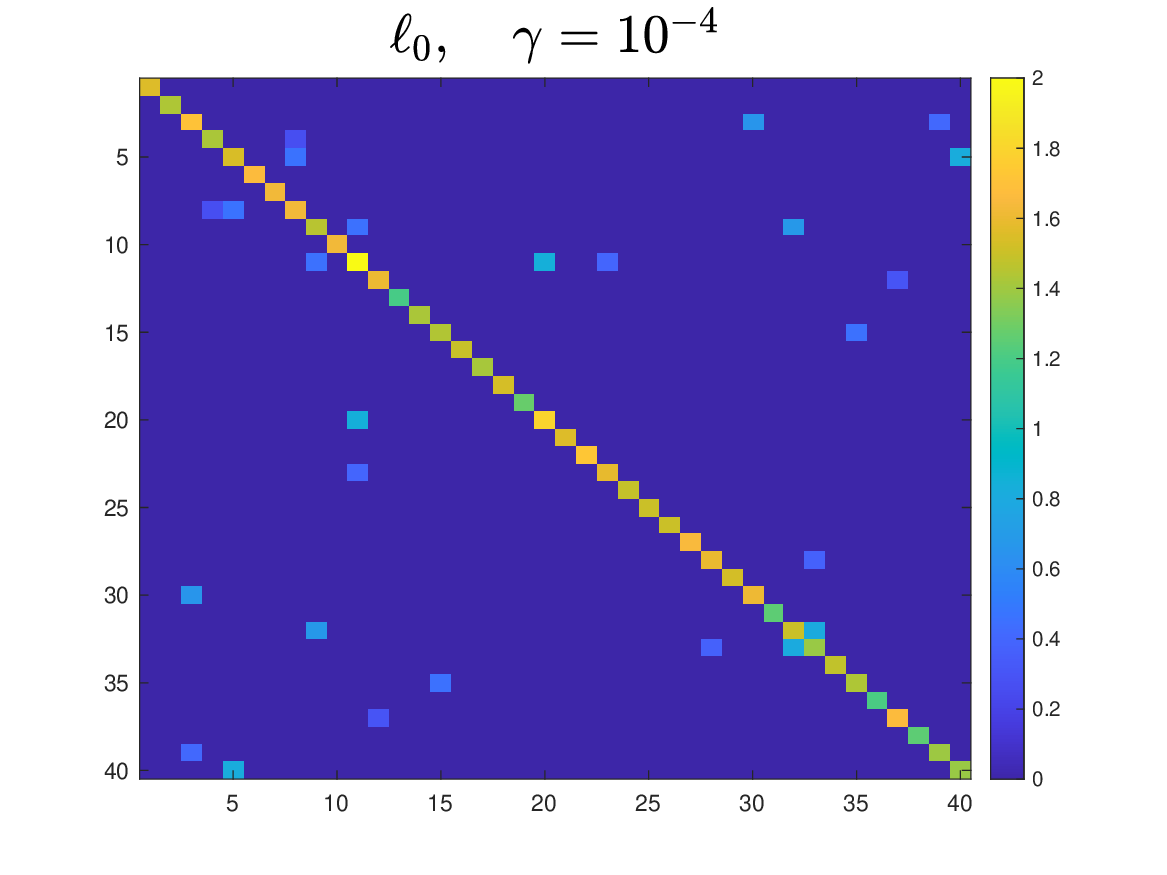}}
	\end{minipage}
	\hfill
	\begin{minipage}[b]{0.19\linewidth}
		\centering
		\centerline{\includegraphics[width=0.95\linewidth]{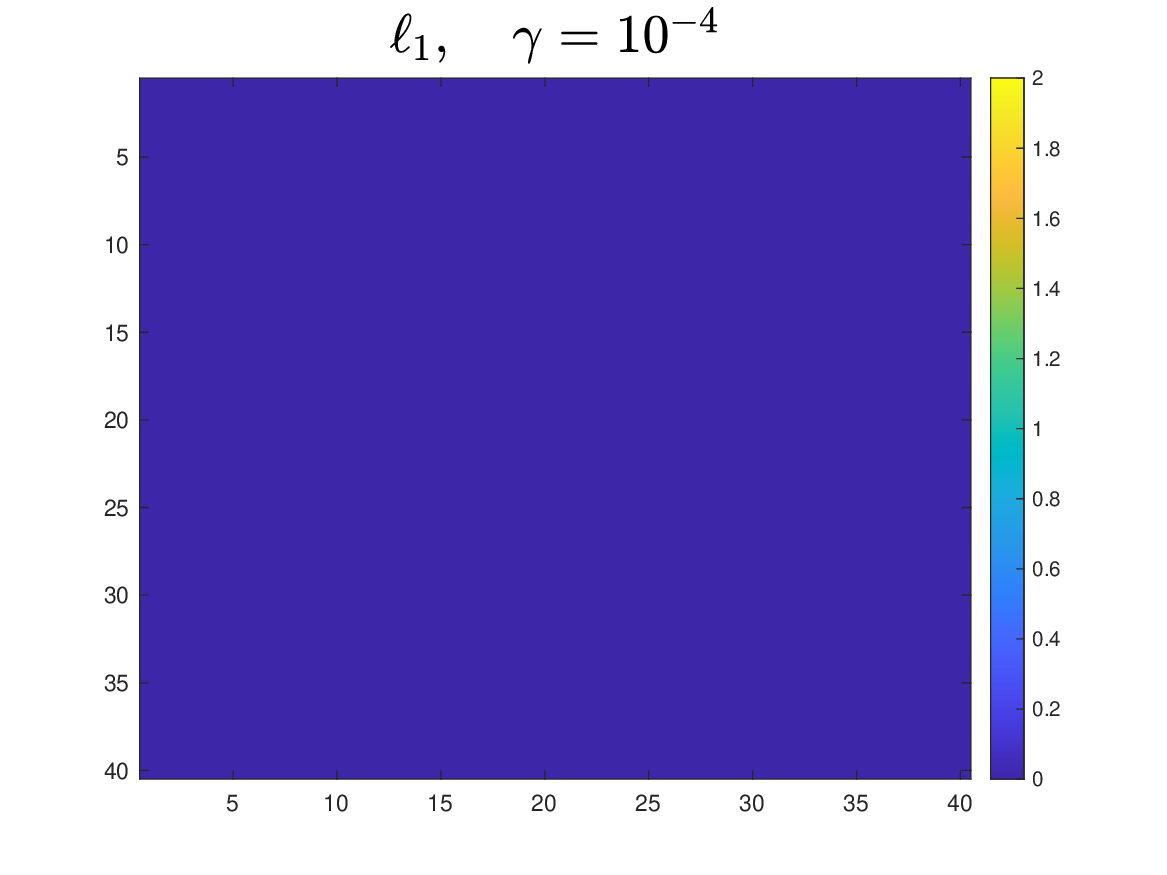}}
	\end{minipage}
	\hfill
	\begin{minipage}[b]{0.19\linewidth}
		\centering
		\centerline{\includegraphics[width=0.95\linewidth]{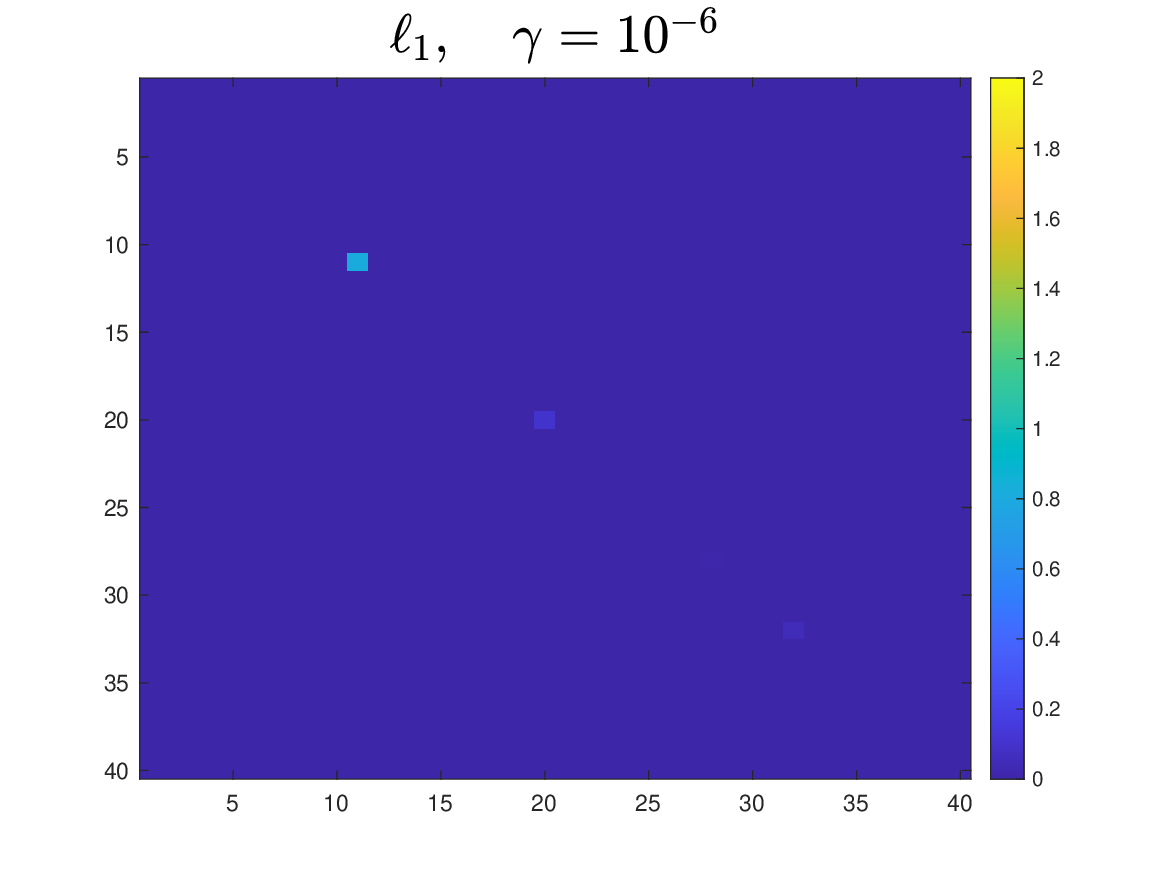}}
	\end{minipage}
	\hfill
	\begin{minipage}[b]{0.19\linewidth}
		\centering
		\centerline{\includegraphics[width=0.95\linewidth]{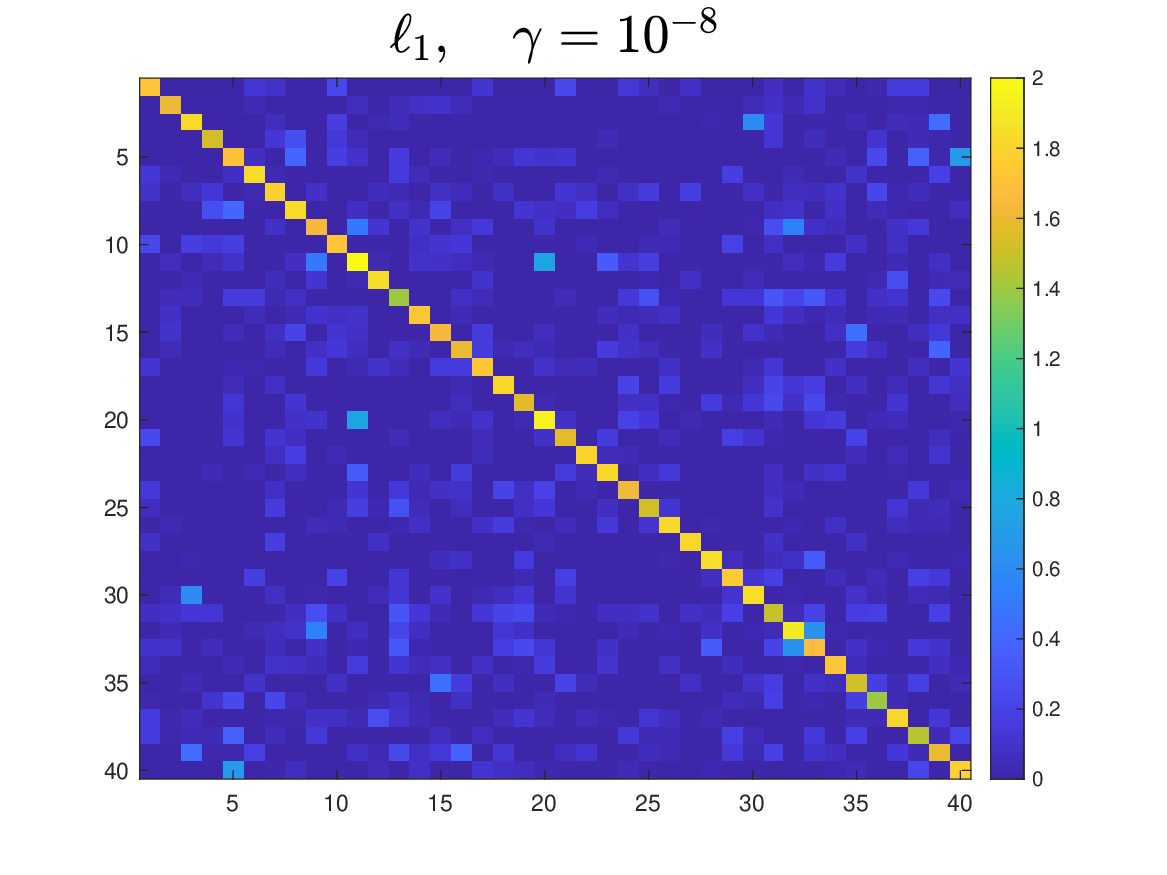}}
	\end{minipage}
	\centerline{\small (b)}
	\caption{\lyn{Comparison of the sparse structure of the estimated $\Sb^*$ under the $\ell_0$ and $\ell_1$ regularizations across different conditions, when $N=1000$, $r=4$ and $\mathrm{SNR}=6$. (a) The true $\hat \Sb$ is diagonal (scaled identity). (b) The true $\hat\Sb$ is nondiagonal sparse.}}
	\label{fig:0_1_compare_total}
\end{figure*}

\subsubsection{Comparison with \cite{ciccone2018factor}}

Similar to the simulation setup in \cite{ciccone2018factor}, \lyn{we set the true sparse component $\hat\Sb$ to be a scaled identity matrix in step i), and  $\SNR=1$. The parameter $\gamma$ is set to $10^{-2}$ in this case. Then,}
we consider two scenarios  in which the true rank of $\Gammab$ is \lyn{$r=4$} and $r=10$, respectively, and perform Monte Carlo simulations \lyn{consisting of $200$ trials} with different data lengths $N$. The subspace recovery index \eqref{ratio} for each trial is shown via the \texttt{boxplot} in Fig.~\ref{fig:ratio_nondiag}.
One can observe that the subspace recovery accuracy increases with the number of available samples. \lyn{Moreover, compared to Fig.~2 in the prior work \cite{ciccone2018factor}, panel (b) of our Fig.~\ref{fig:ratio_nondiag} shows that the method in \cite{ciccone2018factor} performs better in the subspace recovery of $\Gammab$ when the true rank is $r=10$. 
However, panel (a) shows that in the case of $r=4$, the opposite  happens, i.e., our Algorithm~\ref{alg:admm} achieves better performance, even when the outliers (marked with red ``+'') are taken into consideration.}

In addition, we have also generated a set of samples of size \lyn{$N=1000$} and \lyn{$r=4$} to see the effect of parameters $C$, $\mu$, and $\rho$ on the subspace recovery index. More specifically, Fig.~\ref{fig:parameter_cross} illustrates the quantity $\mathrm{ratio}(\mathbf{\Gamma}^{*})$ as a function of $(C, \mu)$ for a fixed $\rho=16$. The surface plot is quite flat which means that Algorithm \ref{alg:admm} achieves robust recovery of the column space of $\Gammab$ under different parameter configurations $(C,\mu)$ which are used in the CV. We add that the figure basically does not change for different values of $\rho$ (with the same range for $(C, \mu)$) and hence they are not reported here.

\begin{figure}[t]
	\begin{minipage}[b]{.47\linewidth}
		\centering
		\centerline{\includegraphics[width=.9\linewidth]{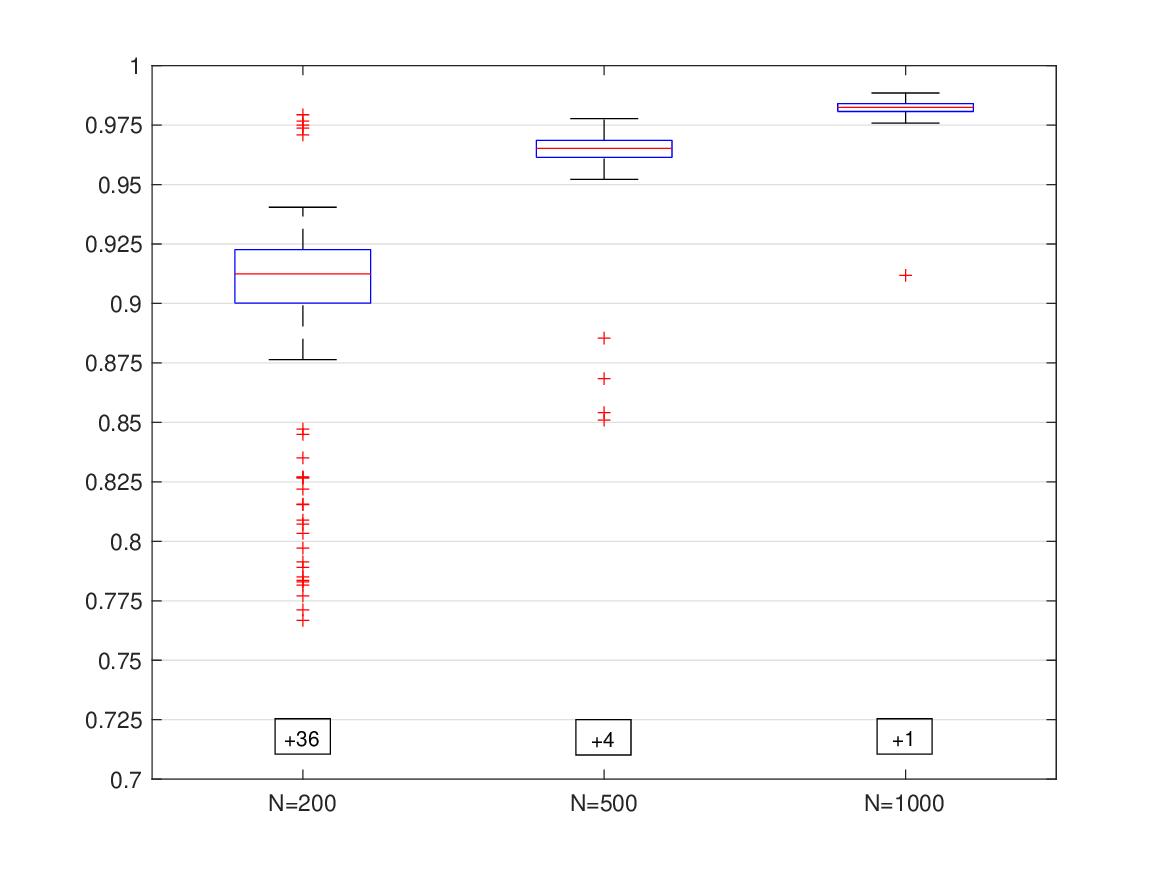}}
		\centerline{(a) \lyn{$r=4$}}
	\end{minipage}
	\hfill
	\begin{minipage}[b]{0.47\linewidth}
		\centering
		\centerline{\includegraphics[width=.9\linewidth]{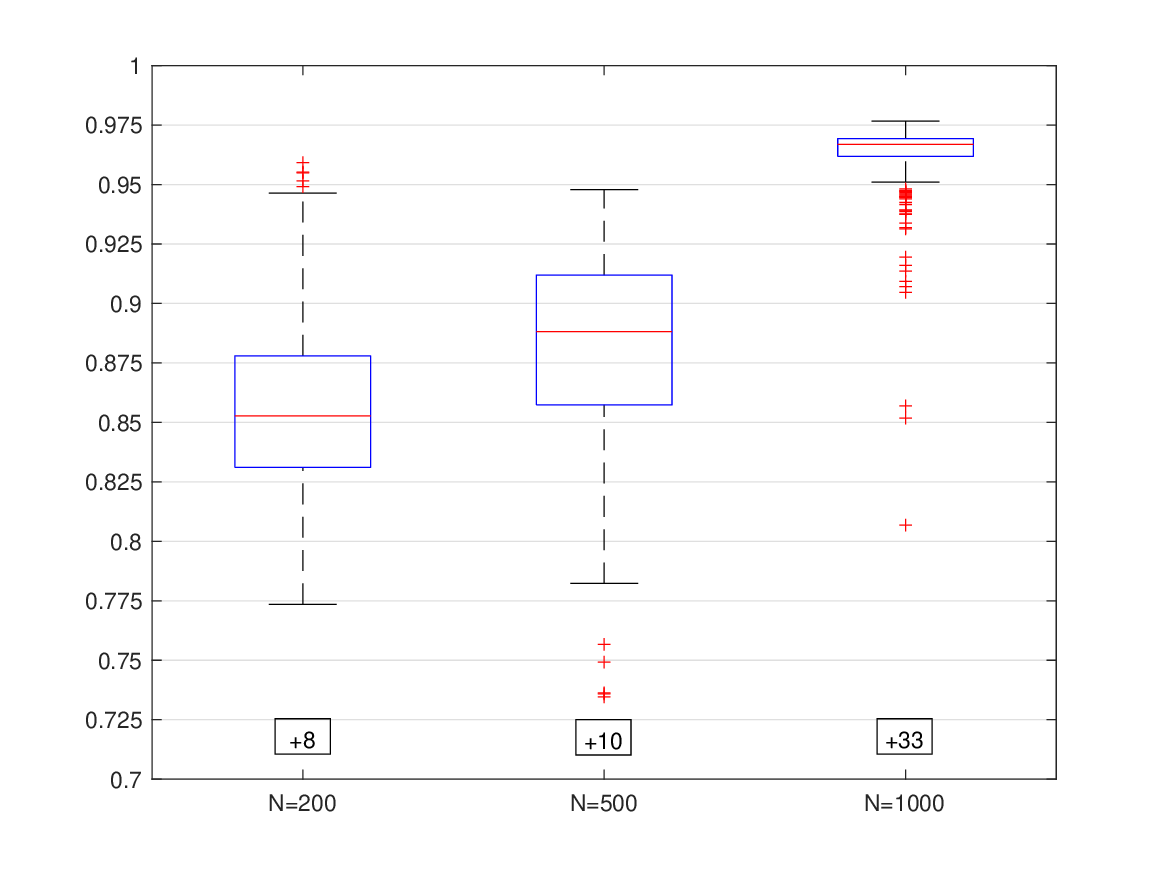}}
		\centerline{(b) $r=10$}
	\end{minipage}
	 \lyn{
	\caption{The subspace recovery accuracy of $\Gammab$ for $N=200$, $500$, and $1000$ in two scenarios: when the true rank $r = 4$ and when $r = 10$, respectively, with the true $\hat\Sb$ being diagonal (scaled identity) and $\mathrm{SNR}=1$. The integer below each boxplot represents the total number of outliers (marked with red ``+'').}
	\label{fig:ratio_nondiag}
	}
\end{figure}

\begin{figure}[t]
	\centering
	\includegraphics[width=.5\linewidth]{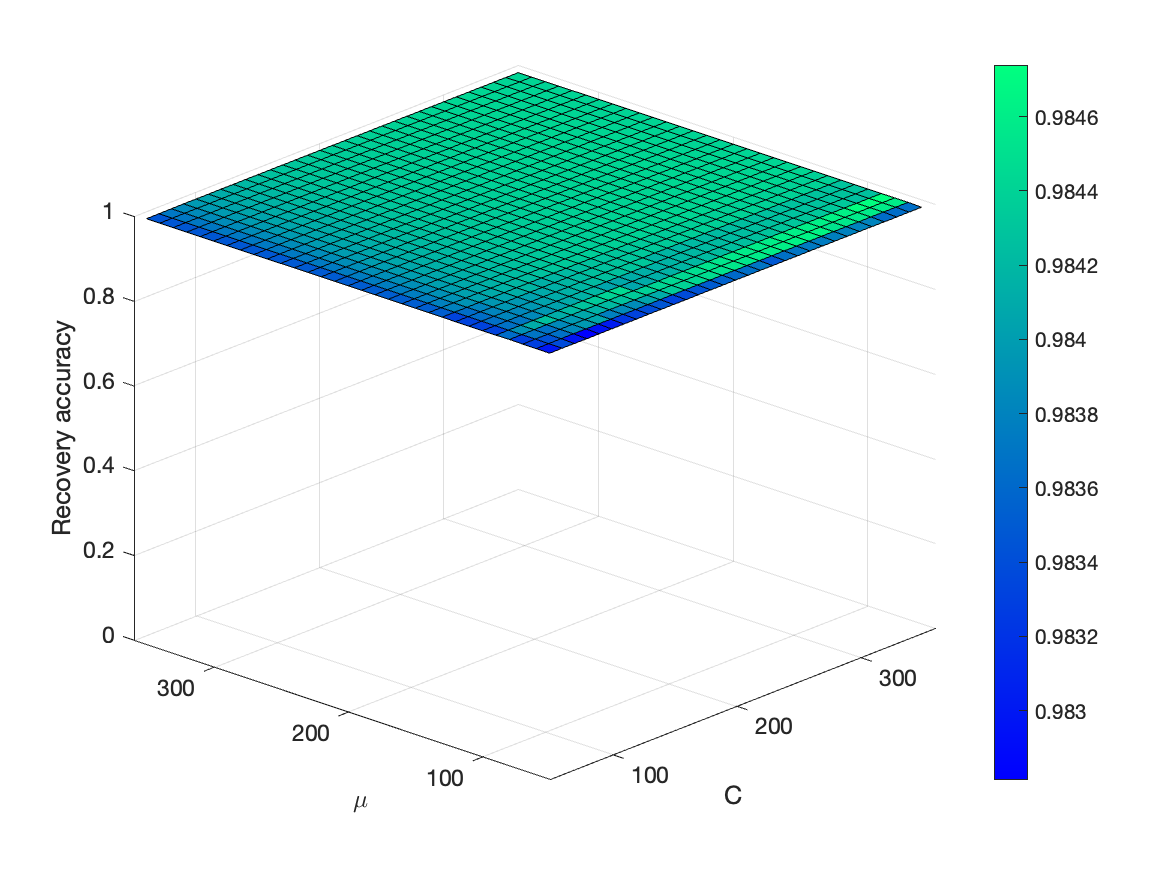}
	\lyn{
	\caption{The subspace recovery accuracy across various parameter configurations $(\mu,C)$ with $\rho=16$ fixed on a dataset of size $N=1000$, the true $r=4$ and $\mathrm{SNR}=1$.}
	\label{fig:parameter_cross}
}
\end{figure}

\subsection{A real data example}

In this subsection, Algorithm \ref{alg:admm} is applied to an economic dataset\footnote{The dataset was sourced from  \url{http://www.stern.nyu.edu/~adamodar/pc/datasets/betas.xls} 
	(downloaded in May 2023).}, which collects the data corresponding to nine financial indicators ($p=9$) across 92 different sectors ($N=92$) of the U.S. economy. These indicators include the Hi-Lo risk, the standard deviation of operating income, the debt/equity ratio, the beta, the unlevered beta, the unlevered beta adjusted for cash, the effective tax rate, the cash/firm value ratio, and the standard deviation of equity. For detailed definitions and calculation of these indicators, we refer the readers to the website written in the footnote.

On this real dataset, since the true factor loading matrix $\Gammab$ and the true noise covariance matrix $\hat\Sb$ are unknown, we are mostly interested in estimating the number of latent factors, i.e., the rank of $\hat\Lb$. 
Similar to simulations on synthetic datasets, we set $\gamma=10^{-4}$, and use the CV to select parameters $C$, $\mu$, and $\rho$ from the candidate sets $\left\{10,15,20,\cdots,60\right\}$, $\left\{10,11,12,\cdots,60\right\}$, and $\left\{2^0, 2^1, \cdots, 2^5\right\}$, respectively. 
The other parameters are identical to those in the synthetic data examples. 
As a result, Algorithm~\ref{alg:admm} returns a rank estimate $r^*=3$. Moreover,
it is important to point out that Algorithm \ref{alg:admm} can robustly estimate the number of latent factors as $3$ under \emph{any} parameter configuration selected from the candidate sets above for the CV\footnote{In this case, we do not carry out the CV procedure for the selection of parameters, but simply take parameters from the candidate sets and run Algorithm~\ref{alg:admm}.}, and \cite{ciccone2018factor} also discussed the rationality of $r^*=3$. In addition, Fig.~\ref{fig:kl} displays the changes in the value of $\Dcal_{\KL} (\Lb^*+\Sb^* \| \check{\Sigmab})$ with respect to the parameters $\rho$ and $\mu$, with $C=20$ fixed. 
Under the same simulation setup, the Alternating Minimization method described in \cite{FA_CDC_24}, which updates the sparse component $\Sb$ using coordinate descent and the low-rank component $\Lb$ using the CVX \cite{cvx,gb08}, can stably return a rank estimate of $3$ \emph{only} when $\mu$ is taken from the subset $\{24, 25, \cdots, 60\}$. Meanwhile, comparing Fig.~\ref{fig:kl} in the present paper with Fig.~4 of \cite{FA_CDC_24} which depicts $\Dcal_{\KL} (\Lb^*+\Sb^* \| \check{\Sigmab})$ as a function of $\mu$ under a fixed $C$,  we see that the refined estimate of the covariance matrix $\Sigmab^*=\Lb^*+\Sb^*$ returned by Algorithm~\ref{alg:admm} appears to stay closer to the sample covariance matrix $\check{\Sigmab}$ as measured by the KL divergence.

\begin{figure}[t]
	\centering
	\includegraphics[width=.5\linewidth]{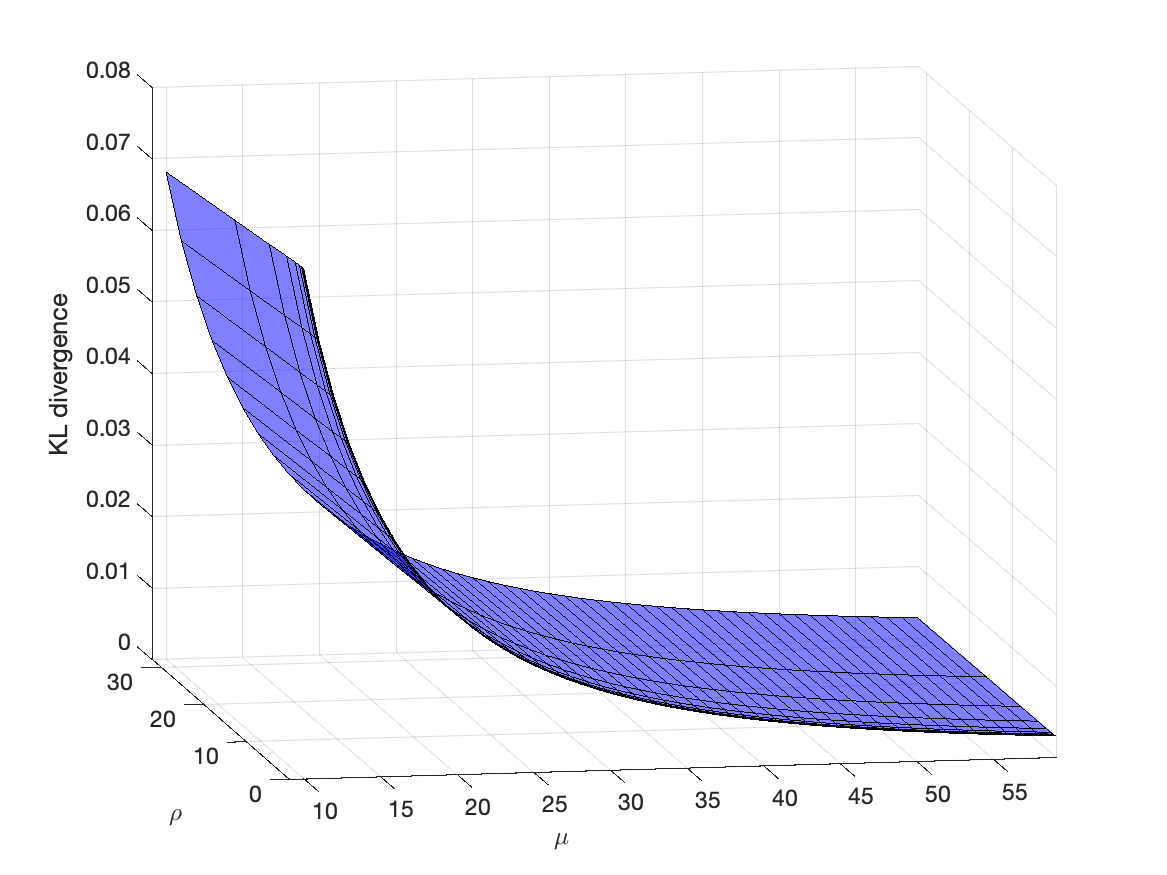}
	\caption{The value of $\Dcal_{\KL} (\Lb^*+\Sb^* \| \check{\Sigmab})$ as a function of $\rho$ and $\mu$ under a fixed $C=20$, where $(\Lb^*, \Sb^*)$ represents the convergent iterate returned by Algorithm~\ref{alg:admm}.} 
	\label{fig:kl}
\end{figure}

\section{Conclusion}\label{sec:conc}

We have formulated an optimization model for Factor Analysis which can be described in terms of low rank plus sparse matrix decomposition.
Although the objective function includes the $\ell_{0}$ norm of a nonconvex nonsmooth nature, we have established an optimality theory using a generalized notion of the KKT point called P-stationary point. 
Furthermore, we have developed an ADMM solver for the optimization problem and have provided convergence analysis of the algorithm.  
At last, results of numerical simulations have demonstrated not only the superiority of our algorithm in subspace recovery of the low-rank component,  but also its ability to explore the unknown structure of the sparse component.
In particular, the number of hidden factors can be estimated in a robust manner as shown in the real data example.

\lyn{For future work, it could be interesting to compare our approach with reweighted $\ell_1$ minimization \cite{candes2008enhancing} and other optimization models involving nonconvex relaxations of the $\ell_0$ norm.}

\appendix

We first recall the definition of a strongly convex function which can be found e.g., in \cite{boyd2004convex}.

\begin{definition}[Strongly convex function]\label{strongly_convex}
	A differentiable function $f:\Rbb^n\rightarrow\Rbb$ is strongly convex, if for any $\xb, \yb \in \Rbb^n$ there exists a constant $m>0$ such that 
\begin{equation}
		f(\yb)\geq f(\xb)+\innerprod{\nabla f(\xb)}{\yb-\xb}+(m/2)\|\yb-\xb\|^2.
\end{equation}
\end{definition}

Following Definition \ref{strongly_convex}, suppose that $\xb^*$ is a minimizer of the strongly convex function $f$ with a constant $m$. Then the inequality
\begin{equation}\label{strong_conv_optim_inequal}
		f(\xb)\geq f(\xb^*)+(m/2)\|\xb-\xb^*\|^2
\end{equation}
holds because $\xb-\xb^*$ is not a descent direction for any feasible $\xb$ and hence $\innerprod{\nabla f(\xb^*)}{\xb-\xb^*}\geq 0$.

{\noindent \it Proof of Lemma~\ref{Sufficient decrease}}.
The Hessian of $\Lcal_\rho(\Ab)$ with respect to $\Lb$ is
$$\nabla_{\Lb}^2 \Lcal_\rho(\Ab)=\mu(\Lb+\Sb)^{-1} \otimes (\Lb+\Sb)^{-1}+\rho \Ib$$
where $\otimes$ denotes the Kronecker product.
Given $\Lb+\Sb\succ 0$ and properties of the Kronecker product \cite{lancaster1985theory}, we have
$$\nabla_{\Lb}^2 \Lcal_\rho(\Ab) \succeq \rho\Ib,$$
 which means that the function $\Lcal_\rho(\Lb,\Sb^k,\Ub^k,\Vb^k;\Lambdab^k,\Thetab^k)$ is $\rho$-strongly convex with respect to $\Lb$.
Therefore, for any $\Lb^k$, the minimizer $\Lb^{k+1}$ of $\Lcal_\rho(\Lb,\Sb^k,\Ub^k,\Vb^k;\Lambdab^k,\Thetab^k)$ satisfies
		\begin{equation}\label{descent_l}
			\begin{aligned}
				&\Lcal_\rho(\Lb^{k+1},\Sb^k,\Ub^k,\Vb^k;\Lambdab^k,\Thetab^k)\\
				\leq & \Lcal_\rho(\Lb^{k},\Sb^k,\Ub^k,\Vb^k;\Lambdab^k,\Thetab^k)-\frac{\rho}{2}\|\Lb^{k+1}-\Lb^k\|_\F^2
			\end{aligned}
		\end{equation}
   in view of \eqref{strong_conv_optim_inequal}.

Next we consider the subproblem \eqref{subprob_S} which can be expressed as 
		\begin{equation}
			\Sb^{k+1}  = \underset{\Sb}{\argmin} \ h(\Sb)+C\|\Sb\|_0,
		\end{equation}
		 where 
		$$h(\Sb) := \innerprod{\mu \check{\Sigmab}^{-1}-\Thetab^k}{\Sb}-\mu\log\det(\Lb^{k+1}+\Sb)+\frac{\rho}{2}\|\Sb-\Vb^k\|_\F^2.$$
Using the fact that  $\|\cdot\|_\F^2$ is a smooth function and following Proposition~\ref{assumption_satisifies}, we can deduce that the gradient of $h(\Sb)$ is Lipschitz continuous with a Lipschitz constant $K+\rho$.
In addition by Lemma \ref{lem_lower_semiconti}, $\|\cdot\|_0$ is a lower semicontinuous function with an obvious lower bound zero.
Since we have chosen $\gamma \leq \frac{1}{K+2\rho}< \frac{1}{K+\rho}$, for any $\Sb^k$, it follows from  \cite[Lemma 2]{bolte2014proximal} that 
		\begin{equation}
			\begin{aligned}
				&h(\Sb^{k+1})+C\|\Sb^{k+1}\|_0 \\
				\leq & h(\Sb^k)+C\|\Sb^k\|_0-\frac{\frac{1}{\gamma}-K-\rho}{2}\|\Sb^{k+1}-\Sb^k\|_\F^2\\
    \leq & h(\Sb^k)+C\|\Sb^k\|_0-\frac{\rho}{2}\|\Sb^{k+1}-\Sb^k\|_\F^2,
			\end{aligned}
		\end{equation}
		namely
		\begin{equation}\label{descent_s}
			\begin{aligned}
				&\Lcal_\rho(\Lb^{k+1},\Sb^{k+1},\Ub^k,\Vb^k;\Lambdab^k,\Thetab^k)\\
				\leq & \Lcal_\rho(\Lb^{k+1},\Sb^k,\Ub^k,\Vb^k;\Lambdab^k,\Thetab^k)-\frac{\rho}{2}\|\Sb^{k+1}-\Sb^k\|_\F^2.
			\end{aligned}
		\end{equation}
		
		Notice also that  $\Lcal_\rho(\Lb^{k+1},\Sb^{k+1},\Ub,\Vb^k;\Lambdab^k,\Thetab^k)$ is $\rho$-strongly convex with respect to $\Ub$. Hence for any $\Ub^k$, we have
		\begin{equation}\label{descent_u}
			\begin{aligned}
				&\Lcal_\rho(\Lb^{k+1},\Sb^{k+1},\Ub^{k+1},\Vb^k;\Lambdab^k,\Thetab^k)\\
				\leq & \Lcal_\rho(\Lb^{k+1},\Sb^{k+1},\Ub^k,\Vb^k;\Lambdab^k,\Thetab^k)-\frac{\rho}{2}\|\Ub^{k+1}-\Ub^k\|_\F^2,
			\end{aligned}
		\end{equation}
		and similarly, 
		\begin{equation}\label{descent_v}
			\begin{aligned}
				&\Lcal_\rho(\Lb^{k+1},\Sb^{k+1},\Ub^{k+1},\Vb^{k+1};\Lambdab^k,\Thetab^k)\\
			\leq & \Lcal_\rho(\Lb^{k+1},\Sb^{k+1},\Ub^{k+1},\Vb^k;\Lambdab^k,\Thetab^k)-\frac{\rho}{2}\|\Vb^{k+1}-\Vb^k\|_\F^2.
			\end{aligned}
		\end{equation}
		
		From the update formulas \eqref{dual_update_Lambda} and \eqref{dual_update_Theta}, it can be calculated that
		\begin{equation}\label{descent_lambda}
			\begin{aligned}
  & \Lcal_\rho(\Lb^{k+1},\Sb^{k+1},\Ub^{k+1},\Vb^{k+1};\Lambdab^{k+1},\Thetab^{k+1}) \\
			= & \Lcal_\rho(\Lb^{k+1},\Sb^{k+1},\Ub^{k+1},\Vb^{k+1};\Lambdab^k,\Thetab^k) \\
				& +\frac{1}{\rho}\|\Lambdab^{k+1}-\Lambdab^k\|_\F^2+\frac{1}{\rho}\|\Thetab^{k+1}-\Thetab^k\|_\F^2.
			\end{aligned}
		\end{equation}

    Combining the formulas \eqref{descent_l}, \eqref{descent_s}, \eqref{descent_u}, \eqref{descent_v}, and \eqref{descent_lambda}, for any $\epsilon\in (0, \rho)$, we obtain
		\begin{equation}\label{descent_Sum}
			\begin{aligned}
				& \Lcal_\rho(\Ab^{k})-\Lcal_\rho(\Ab^{k+1})\\
			 \geq & \frac{\rho}{2} (\star)  - \frac{1}{\rho} \left( \|\Lambdab^{k+1}-\Lambdab^k\|_\F^2 + \|\Thetab^{k+1}-\Thetab^k\|_\F^2 \right) \\
			= &  \frac{\epsilon}{2} (\star )+	\frac{1}{\rho} \left(\|\Lambdab^{k+1}-\Lambdab^k\|_\F^2+\|\Thetab^{k+1}-\Thetab^k\|_\F^2 \right)\\
				& +\frac{\rho-\epsilon}{2} (\star)
				-\frac{2}{\rho} \left(\|\Lambdab^{k+1}-\Lambdab^k\|_\F^2+\|\Thetab^{k+1}-\Thetab^k\|_\F^2 \right),
			\end{aligned}
		\end{equation} 
		where $(\star) := \|\Lb^{k+1}-\Lb^k\|_\F^2+\|\Sb^{k+1}-\Sb^k\|_\F^2+\|\Ub^{k+1}-\Ub^k\|_\F^2+\|\Vb^{k+1}-\Vb^k\|_\F^2$.
  It remains to show that the terms in the last line of \eqref{descent_Sum} are nonnegative for some suitable choice of $\rho$ and $\epsilon$.
  To this end, we impose that
  \begin{equation}\label{cond_rho}
			\rho >  \sup_{k\geq 1} \frac{2\left( \|\Lambdab^{k+1}-\Lambdab^k\|_\F^2 + \|\Thetab^{k+1}-\Thetab^k\|_\F^2\right)^{1/2}}{(\star)^{1/2}}.
  \end{equation}	
  Due to the strict inequality, there must exist a sufficiently small $\epsilon\in (0, \rho)$ such that the left hand side of \eqref{cond_rho} can be replaced by $\rho-\epsilon$. These two strict inequalities imply that
		 $$
		 \frac{\rho-\epsilon}{2}(\star) \geq \frac{2}{\rho} \left(\|\Lambdab^{k+1}-\Lambdab^k\|_\F^2+\|\Thetab^{k+1}-\Thetab^k\|_\F^2 \right)
		 $$
   holds for all $k\geq 1$. Thus we complete the proof.

\bibliographystyle{ieeetr}
\bibliography{refs}

\begin{IEEEbiography}[{\includegraphics[width=1in,height=1.25in,clip,keepaspectratio]{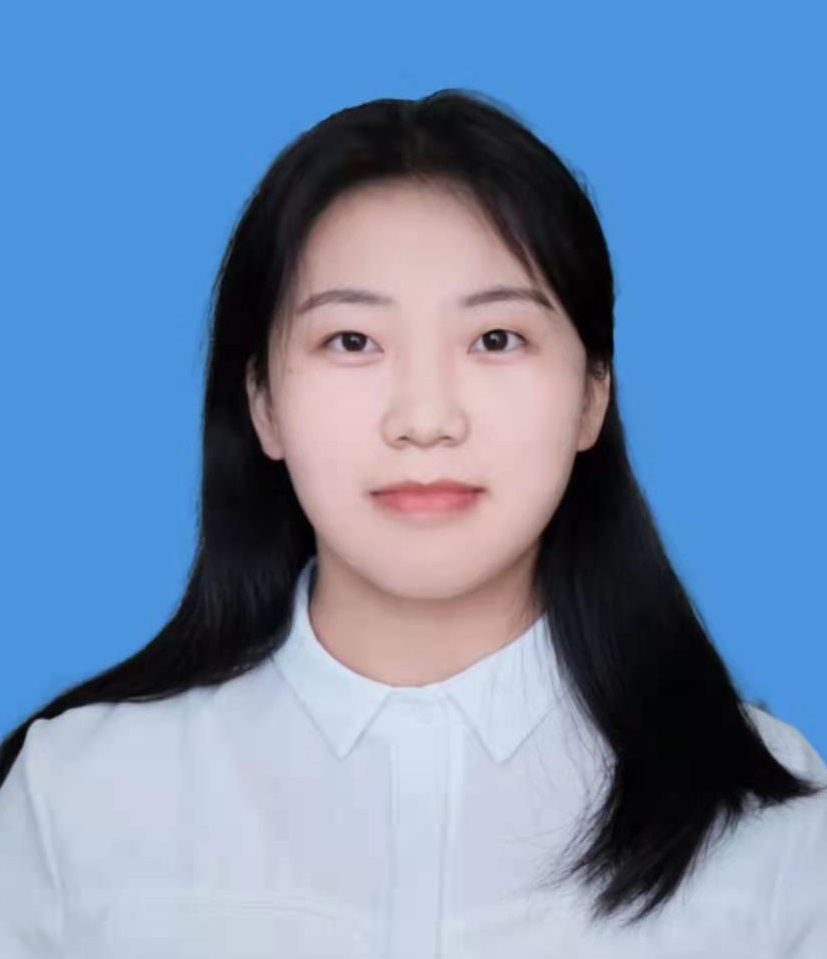}}]{Linyang Wang} (Student Member, IEEE) received the B.S. degree in Applied Mathematics from Qufu Normal University, China, in 2018, the M.S. degree in Operation Research and Control Theory from the Beijing University of Technology, in 2022. She is currently pursuing the Ph.D. degree in the School of Intelligent Systems Engineering, Sun Yat-sen University. Her current research interests include systems identification, signal processing, and sparse optimization.
\end{IEEEbiography}

\vskip -2.5\baselineskip plus -1fil

\begin{IEEEbiography}[{\includegraphics[width=1in,height=1.25in,clip,keepaspectratio]{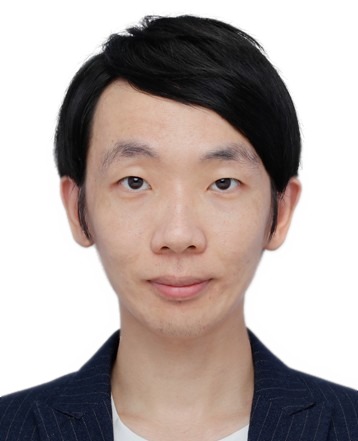}}]{Bin Zhu} (Member, IEEE) received a B.Eng. degree from Xi'an Jiaotong University, Xi'an, China in 2012 and a M.Eng. degree from Shanghai Jiao Tong University, Shanghai, China in 2015, both in control science and engineering. In 2019, he obtained a Ph.D. degree in information engineering from University of Padova, Padova, Italy, and he had a one-year postdoc position in the same university. Since December 2019, he has been working at the School of Intelligent Systems Engineering, Sun Yat-sen University, Shenzhen, China, where he is now an associate professor. His current research interest includes spectral analysis, frequency estimation, and sparsity-promoting techniques for system identification, signal processing, and machine learning.
	
	Dr. Zhu has been appointed as an Associate Editor of the Editorial Board of the International Conference on System Theory, Control and Computing (ICSTCC) since December 2023.
\end{IEEEbiography}

\vskip -2.5\baselineskip plus -1fil

\begin{IEEEbiography}[{\includegraphics[width=1in,height=1.25in,clip,keepaspectratio]{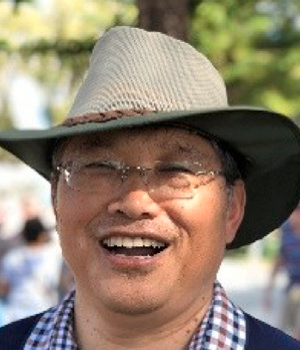}}]{Wanquan Liu} (Senior Member, IEEE) received the B.S. degree in applied mathematics from Qufu Normal University, China, in 1985, the M.S. degree in control theory and operation research from the Chinese Academy of Sciences in 1988, and the Ph.D. degree in electrical engineering from Shanghai Jiao Tong University in 1993. He is currently a Full Professor with the School of Intelligent Systems Engineering, Sun Yat-sen University, Shenzhen, China. He once held the ARC Fellowship, the U2000 Fellowship, and the JSPS Fellowship, and attracted research funds from different resources over 2.4 million dollars. His current research interests include large-scale pattern recognition, signal processing, machine learning, and control systems.
\end{IEEEbiography}

\end{document}